\documentclass[11p]{amsart}
\usepackage{mathtools}
\usepackage{amssymb}
\usepackage{amsfonts}
\usepackage{amsthm}
\usepackage{mathrsfs}
\usepackage{polynom}
\usepackage{mathbbol}
\usepackage{latexsym,amscd}
\usepackage[all]{xy}
\usepackage{color}

\newtheorem{thm}{Theorem}
\newtheorem{prop}{Proposition}
\newtheorem{lem}{Lemma}
\numberwithin{equation}{section}
\numberwithin{prop}{section}
\numberwithin{lem}{section}
\numberwithin{thm}{section}
\newtheorem{cor}{Corollary}
\numberwithin{cor}{section}

\theoremstyle{definition}
\newtheorem{defn}{Definition}
\numberwithin{defn}{section}

\usepackage{mathtools}
\usepackage{tikz}
\usetikzlibrary{chains}
\usepackage{graphics}
\usepackage{epic,eepic}

\newtheorem{rem}{Remark}
\numberwithin{rem}{section}

\newcommand{\nop}[1]{{}^{\scriptscriptstyle{\circ}}_{\scriptscriptstyle{\circ}}{#1}{}^{\scriptscriptstyle{\circ}}_{\scriptscriptstyle{\circ}}}

\def \<{\langle}
\def \>{\rangle}

\def \a{\alpha }

\def \b{\beta }

\newcommand{\bea}{\begin{eqnarray}}
\newcommand{\eea}{\end{eqnarray}}
\newcommand{\be}{\begin {equation}}
\newcommand{\ee}{\end{equation}}
\newcommand{\g}{\mathfrak{g}}

\newcommand{\h}{\mathfrak{h}}
\newcommand{\wt}{{\rm {wt} }   }

\newcommand{\xbet}[1]{x^{\hnu}_{\beta}\left(#1\right)}
\newcommand{\xbbet}[2]{x^{\hnu}_{\beta_{#1}}\left(#2\right)}
\newcommand{\xa}[2]{x^{\hat{\nu}}_{\alpha_{#1}}\left(#2\right)}
\newcommand{\xaa}[3]{x^{\hat{\nu}}_{\alpha_{#1}+\alpha_{#2}}\left(#3\right)}

\newcommand{\xb}[1]{x^{\hat{\nu}}_{\b}\left(#1\right)}

\newcommand{\hnu}{\hat{\nu}}
\newcommand{\n}{\mathfrak{n}}
\newcommand{\vL}{v_{\Lambda}}
\newcommand{\ta}[1]{\tau_{\gamma_{#1},\theta_{#1}}}
\newcommand{\on}{\overline{\n}[\hnu]}
\newcommand{\ps}[1]{\psi_{\gamma_{#1},\theta_{#1}}}

\newcommand{\un}[0]{U\left(\overline{\mathfrak{n}}[\hnu]\right)}

\newcommand{\CC}[0]{\mathbb{C}}
\newcommand{\ZZ}[0]{\mathbb{Z}}

\newcommand{\EE}[0]{{\varepsilon}}

\tikzset{node distance=2em, ch/.style={circle,draw,on chain,inner sep=2pt},chj/.style={ch,join},every path/.style={shorten >=4pt,shorten <=4pt},line width=1pt,baseline=-1ex}

\let\dlabel=\alabel

\newcommand{\dnode}[2][chj]{%
\node[#1,label={below:\dlabel{#2}}] {};
}

\newcommand{\dnodebr}[1]{%
\node[chj,label={below right:\dlabel{#1}}] {};
}

\newcommand{\dydots}{%
\node[chj,draw=none,inner sep=1pt] {\dots};
}

\begin{document}
\title{Vertex-algebraic structure of principal subspaces of the basic modules for twisted Affine Kac-Moody Lie algebras of type $A_{2n-1}^{(2)}, D_n^{(2)}, E_6^{(2)}$}  
\author{Michael Penn and Christopher Sadowski}

\begin{abstract} We obtain a presentation of principal subspaces of the basic modules for the twisted affine Kac-Moody Lie algebras of type  $A_{2n-1}^{(2)}$, $D_n^{(2)}$ and $E_6^{(2)}$. Using this presentation, we construct exact sequences among these principal subspaces, and use these exact sequences to obtain recursions satisfied by graded dimensions of the principal subspaces. Solving these recursions, we obtain the graded dimensions of the principal subspaces.

\end{abstract}

\maketitle

\section{Introduction}
Principal subspaces of standard modules for untwisted affine Lie algebras were introduced and studied by Feigin and Stoyanovsky in \cite{FS1}--\cite{FS2}. Motivateded by work of Lepowsky and Primc \cite{LP}, Feigin and Stoyanovsky noticed difference-two type partition conditions in these structures. In particular, Feigin and Stoynovsky showed that the multigraded dimensions of the principal subspace of the basic modules for $A_1^{(1)}$ are the sum-sides of the Rogers-Ramanujan identities, and more generally, the multigraded dimensions of the principal subspaces of the higher level standard modules for $A_1^{(1)}$ are the sum-sides of the Andrews-Gordon identities. 

In \cite{CLM1}--\cite{CLM2}, Capparelli, Lepowsky, and Milas, studied the principal subspaces of the standard modules for $A_1^{(1)}$ using the theory of vertex operator algebras and intertwining operators. Under certain assumptions (namely, presentations of principal subspaces using generators and relations used in \cite{FS1}--\cite{FS2}), they constructed exact sequences among these principal subspaces. They used these exact sequences to show that the multigraded dimensions of the principal spaces satisfy the Rogers-Ramanujan and Rogers-Selberg recursions. From these recursions, they recovered the graded dimensions found in \cite{FS1}--\cite{FS2}. In \cite{CalLM1}--\cite{CalLM2}, Calinescu, Lepowsky, and Milas, again using the theory of vertex operator algebras and intertwining operators, proved the presentations used in \cite{FS1}--\cite{FS2} and \cite{CLM1}--\cite{CLM2}. In \cite{CalLM3}, they extended their work to study the principal subspaces of basic modules for untwisted affine Lie algebras of type $A,D,E$. They also constructed exact sequences among the principal subspaces, found the recursions that the multigraded dimensions satisfy, and solved those recursions to find the multigraded dimensions of the principal subspaces. 

Principal subspaces have been studied from a vertex-algebraic point of view in many other works. We give a brief overview of recent work in this area. In \cite{C1} and \cite{S1}, principal subspaces of higher level standard $A_2^{(1)}$-modules have been studied in a spirit similar to \cite{CLM1}--\cite{CLM2} and \cite{CalLM1}--\cite{CalLM2}. These principal subspaces were also student in \cite{AKS}, \cite{FFJMM}, and in other works by these authors. Georgiev used the theory of vertex operator algebras and intertwining operators to construct combinatorial bases for principal subspaces of certain higher level standard modules for untwisted affine Lie algebras of type $A,D,E$ in \cite{G}, and this work has since been extended to untwisted affine Lie algebras of type $B,C$ by Butorac in \cite{Bu1}--\cite{Bu2} and to the quantum group case by Ko\v{z}i\'{c} in \cite{Ko}. Principal subspaces of standard modules for more general lattice vertex operator algebras have been studied in \cite{MPe}, \cite{P}, and \cite{Ka1}--\cite{Ka2}. Certain commutative variants of principal subspaces, called Feigin-Stoyanovsky type subspaces, have been studied in \cite{Ba}, \cite{J1}--\cite{J2}, \cite{JP}, and \cite{T1}--\cite{T3}. The multigraded dimensions of principal subspaces are Nahm sums (cf. \cite{Za}) or closely resemble Nahm sums, and modularity properties of these sums have been studied in \cite{BCFK} for the untwisted affine Lie algebras of type $A$ and in \cite{CalMPe} for the twisted affine Lie algebra $A_4^{(2)}$.

The present work is an extension of the work on principal subspaces of basic modules for twisted affine Lie algebras intiated in \cite{CalLM4} and continued in \cite{CalMPe} and \cite{PS}. We study the principal subspaces of the basic modules for twisted affine Lie algebras of type $A_{2n11}^{(2)}$ for $n\ge 2$, $D_n^{(2)}$ for $n\ge 4$, and $E_6^{(2)}$. We prove a presentation of these principal subspaces and use this presentation to find their multigraded dimensions. It is important to note that the ideas involved in this work and in \cite{CalLM4}, \cite{CalMPe}, and \cite{PS} all require slightly different constructions and technical lemmas, and thus have been studied on a case-by-case basis in these works. To date, there is no known way to handle all the cases simultaneously. We also note here that in this paper, as well as in \cite{PS}, the multigraded dimensions of the principal subspaces closely resemble Nahm sums, but are not Nahm sums themselves due to the form that the denominator of the summands (namely, the appearance of terms of the form $(q^2; q^2)_n$ in the denominators).  

We now give a brief outline of this paper. In Section 2 of, we follow \cite{L1} and \cite{CalLM4} and give the vertex-algebraic construction of the basic modules $V_L^T$ for the twisted affine Lie algebras $A_{2n-1}^{(2)}, D_n^{(2)}$, and $E_6^{(2)}$, viewed as twisted modules for the lattice vertex operator algebras $V_L$ of type $A_{2n-1}^{(1)}, D_n^{(1)}$, and $E_6^{(1)}$. This construction uses the Dynkin diagram automorphisms for the finite dimensional Lie algebras of type $A,D,E$, which we denote $\nu$. In Section 3, we recall the notion of principal subspace (cf. \cite{FS1}--\cite{FS2}, \cite{CLM1}--\cite{CLM2}, \cite{CalLM1}--\cite{CalLM4}, and many others) and define the principal subspaces $W_L^T$ of the basic modules constructed in Section 2. We also define certain natural operators which annihilate the highest weight vectors of the basic modules, as well as left ideals generated by these operators. In Section 4, we define certain natural ``shifting" maps among these ideals, and prove various containment properties of these maps. In Section 5, we use these maps to show that the left ideals defined in Section 3 give presentations of our principal subspaces. We use these presentations to construct exact sequences among our principal subspaces, find recursions satisfied by their multigraded dimensions:
\begin{equation}
\chi'(\mathbf{x};q) = \chi'(x_1,\dots,x_{i-1},q^2x_i,x_{i+1},\dots,x_k;q) + x_iq^2 \chi'(q^{2\left<(\a_{i_1})_{(0)}(\a_{i})_{(0)}\right>}x_1,\dots,q^{2\left<(\a_{i_k})_{(0)}(\a_{i})_{(0)}\right>}x_k;q)
\end{equation}
if $\nu\a_i=\a_i$ and
\begin{equation}
\chi'(\mathbf{x};q) = \chi'(x_1,\dots,x_{i-1},qx_i,x_{i+1},\dots,x_k;q) + x_iq \chi'(q^{2\left<(\a_{i_1})_{(0)}(\a_{i})_{(0)}\right>}x_1,\dots,q^{2\left<(\a_{i_k})_{(0)}(\a_{i})_{(0)}\right>}x_k;q)
\end{equation}
if $\nu\a_i\neq\a_i$. Here, in each case, we denote the simple roots by $\alpha_i$, and denote by $(\a_{i})_{(0)}$ the projection of the root $\a_{i}$ onto the $0$ eigenspace of $\nu$ extended to the Cartan subalgebra of the underlying finite dimensional Lie algebra.
Consequently, solving these recursions, we obtain the multigraded dimensions of the principal subspaces:
\begin{equation}
\chi^{'}(\mathbf{x};q) = \sum_{{\bf m} \in (\mathbb{Z}_{\ge 0}^k)}\frac{q^\frac{{\bf m}^t A {\bf m}}{2}}{(q^{a_1};q^{a_1})_{m_1} \cdots (q^{a_k};q^{a_k})_{m_k} }x_1^{m_1}\cdots x_k^{m_k}
\end{equation}
where $a_j=2$ if $\nu\a_{i_j}=\a_{i_j}$, $a_j=1$ if $\nu\a_{i_j}\neq\a_{i_j}$ and $A=2\left(\left<(\a_{i_j})_{(0)},(\a_{i_k})_{(0)}\right>\right)$.
The ideas in Sections 3, 4, and 5 are natural analogues of ideas used in \cite{CalLM1}--\cite{CalLM4} and \cite{PS} to prove similar presentations, adapted to the setting of this paper. 
\section{The Setting}

Let
\be
L=\ZZ\a_1\oplus\cdots\ZZ\a_{l}
\ee
be an integral lattice with associated symmetric non-degenerate bilinear form $\left<\cdot,\cdot\right>:L\times L\to\ZZ$ such that $\left<\a,\a\right>\in 2\ZZ$ for all $\a\in L$. In this setting we often say $L$ is an even lattice. We also assume that $L$ admits an isometry $\nu:L\to L$ such that
\be\label{isomcond1}
\nu^k=1\ee
and
\be\label{isomcond}
\left<\nu^\frac{k}{2}\a,\a\right>\in 2\ZZ~~\text{ for all } \a\in L\ee
if $k$ is even.
In the language of \cite{L1},\cite{CalLM4} and others, we take $k=2$, and in addition assume that for all $\a\in L$ we have $\left<\nu^{k/2}\a,\a\right>\in 2\ZZ$.

\begin{rem} The root lattices of the Lie algebras $A_{2n-1}$, $D_n$, $E_6$ are special cases of this general set-up. We will exploit this to analyze the principal subspaces of the level one standard modules of $A_{2n-1}^{(2)}$, $D_n^{(2)}$ and $E_6^{(2)}$. The $A_{2n}^{(2)}$ case was
examined in \cite{CalMPe}, and required $k=4$ as the period of $\nu$ (which had order $2$), doubling the order of the Dynkin diagram automorphism for the Lie algebra $A_{2n}$.\end{rem}

Following Section 2 of \cite{CalLM4} (also \cite{L1}) while specializing to our setting, we fix $\eta=-1$, the primitive second root of unity, and set $\eta_0=(-1)^2\eta=-1$. Define the functions $C_0:L\times L\to\CC^{\times}$ and $C:L\times L\to\CC^{\times}$ by 
\be
C_0(\a,\b)=(-1)^{\left<\a,\b\right>}\ee
and
\be
C(\a,\b)=\prod_{j=0}^{1}(-(-1)^j)^{\left<\nu^j\a,\b\right>}=(-1)^{\left<\a,\b\right>}.\ee
Here we have $C=C_0$.

We define the central extension $\hat{L}$ of $L$ by $\left<-1\right>$ using the commutator map $C_0\ (=C)$. In other words, we have an exact sequence
\be
1\longrightarrow\left<-1\right>\longrightarrow\hat{L}\overset{\overline{~~}}{\longrightarrow} L\longrightarrow 1
\ee
so that $aba^{-1}b^{-1}=C_0(\overline{a},\overline{b})$ for $a,b\in\hat{L}$. 

\begin{rem}In previous related work (\cite{CalLM4}, \cite{CalMPe}, and \cite{PS}), $C_0$ and $C$ were inequivalent maps and thus determined two inequivalent central extensions $\hat{L}$ and $\hat{L}_{\nu}$, respectively. \end{rem}

After \cite{L1} and \cite{CalLM4}, we let 
\begin{align*}
e:L&\to\hat{L}\\
\a&\mapsto e_{\a}\end{align*}
be a normalized section of $\hat{L}$ so that 
\be
e_0=1\ee
and
\be\overline{e_{\a}}=\a~\text{ for all }~\a\in L.\ee
We define $\epsilon_{C_0}:L\times L\to\left<-1\right>$ by the condition
\be e_{\a}e_{\b}=\epsilon_{C_0}(\a,\b)e_{\a+\b}~ \text{ for all }~\a,\b\in L.\ee

With this set-up $\epsilon_{C_0}(=\epsilon_C)$ is a normalized 2-cocycle associated to the commutator map $C_0(=C)$. In other words,

\begin{align*}
\epsilon_{C_0}(\a,\b)\epsilon_{C_0}(\a+\b,\gamma)&=\epsilon_{C_0}(\b,\gamma)\epsilon_{C_0}(\a,\b+\gamma)\\
\epsilon_{C_0}(0,0)&=1\end{align*}

and

\be
\frac{\epsilon_{C_0}(\a,\b)}{\epsilon_{C_0}(\b,\a)}=C_0(\a,\b).\ee

We choose our 2-cocycle to be the $\mathbb{Z}$-bilinear map determined by
\be
\epsilon_{C_0}(\a_i,\a_j)=\left\{
     \begin{array}{lr}
       1 & \text{for } i\leq j\\
       (-1)^{\left<\a_i,\a_j\right>} & \text{for } i\geq j.
     \end{array}
   \right.\ee

As in \cite{L1} and \cite{CalLM4}, we lift the isometry $\nu$ of $L$ to an automorphism $\hnu$ of $\hat{L}$ such that 
   \be\label{cond1}
   \overline{\hnu a}=\nu\overline{a} \hspace{0.4in} \text{ for } \hspace{0.4in} a\in\hat{L}.
   \ee
and choose $\hnu$ so that 
   \be\label{cond2}
   \hnu a=a \hspace{0.4in} \text{ if } \hspace{0.4in} \nu\overline{a}=\overline{a},
   \ee
   and thus $\hnu^2=1$.  Set 
   \be
   \hnu e_{\a}=\psi(\a)e_{\nu\a}\ee
   where $\psi:L\to \left<-1\right>$ is defined by
   \begin{equation}\label{psilift}
\psi(\alpha)=\begin{cases} 
      \epsilon_{C_0}(\a,\a) & \text{if $L$ is type $A_{2n-1}$} \\
      1 & \text{if $L$ is type $D_{n}$} \\
      (-1)^{r_3r_6}\epsilon_{C_0}(\a,\a) & \text{if $L$ is type $E_6$ and $\a=\sum_{i=1}^6r_i\a_i$}.
   \end{cases}
\end{equation}
Checking that $\hnu(e_{\a}e_{\beta})=\hnu e_{\a}\hnu e_{\beta}$ is clear given that 
\be
\epsilon_{C_0}(\nu\a,\nu\beta)=\begin{cases}
\epsilon_{C_0}(\beta,\a)& \text{if $L$ is type $A_{2n-1}$} \\
      \epsilon_{C_0}(\a,\beta) & \text{if $L$ is type $D_{n}$} \\
      (-1)^{r_6s_3+r_3s_6}\epsilon_{C_0}(\beta,\a) & \text{if $L$ is type $E_6$, $\a=\sum_{i=1}^6r_i\a_i$, and $\b=\sum_{i=1}^6s_i\a_i$}.
   \end{cases}
\ee
Observe that this choice of lifting satisfies (\ref{cond1}), (\ref{cond2}), and 
\be \label{liftingdef} \hnu e_{\a_i}=e_{\nu\a_i}, \ee
where $\a_i$ is a simple root.

\color{black}
   
 Now we recall the construction of the vertex operator algebra $V_{L}$ while lifting $\hnu$ to an automorphism of $V_L$ which we also denote by $\hnu$, and use the space $\hat{L}_{\nu}$ to construct $\hnu$-twisted modules of $V_L$. This closely follows \cite{L1} and \cite{FLM2}. We begin by setting
\be\label{cartan}
\h=\CC\otimes_{\ZZ}L\ee
and extend the bilinear form on $L$ to a $\CC$-bilinear form on $\h$. We view $\h$ as an abelian Lie algebra and consider its affinization
\be
\hat{\h}=\h\otimes \CC[t,t^{-1}]\oplus \CC k\ee
where 
\be
[\a\otimes t^r,\beta\otimes t^s]=\left<\a,\beta\right>r\delta_{r+s,0}k
\ee
for $\a,\beta\in \h$ and $r,s\in\ZZ$, where $k$ is central. We also endow $\hat{\h}$ with a natural $\ZZ$-grading given by $\text{wt}(\a\otimes t^m)=-m$ and $\text{wt}(k)=0$ for $\a\in\h$ and $m\in\ZZ$. Now form the module
\be
M(1)=U(\hat{\h})\otimes_{U(\h\otimes\CC[t]\oplus\CC k)}\CC\cong S(\hat{\h}^{-})~~~~\text{linearly}\ee
where $\hat{\h}^{-}=\h\otimes t^{-1}\CC[t^{-1}]$. Observe that $M(1)$ inherits the $\ZZ$-grading from $\hat{\h}$.

Next we form the $\hat{L}$-module
\be
\CC\{L\}=\CC[\hat{L}]\otimes_{\CC[\left<-1\right>]}\CC\cong\CC[L]~~~\text{(linearly)}.
\ee
Using the inclusion $\iota:\hat{L}\to\CC\{L\}$ we write $\iota(a)=a\otimes 1$. This module is $\CC$-graded so that 
\be
\text{wt}(\iota(1))=0\ee
and
\be\text{wt}(e_{\a})=\frac{1}{2}\left<\a,\a\right>,\ee
for $\a\in L$.
We set
\be\begin{aligned}\label{VL}
V_L&=M(1)\otimes\CC\{L\}\\
&=S\left(\hat{\h}^{-}\right)\otimes \CC[L]~~~\text{(linearly)}\end{aligned}\ee
which is naturally acted upon by $\hat{L}$, $\hat{\h}_{\ZZ}$, $\h$, and $x^h$ $(h\in\h)$ as described in previous work (\cite{CalLM4},\cite{L1},\cite{FLM2}).

For $v=h_1(-m_1)\cdots h_r(-m_r)\otimes\iota(a)\in V_L$ with $h_i\in \h$, $m_i\in\mathbb{N}$, and $a\in\hat{L}$ we set
\be
Y(v,x)=\nop{\frac{1}{(m_1-1)!}\left(\frac{d}{dx}\right)^{m_1-1}h_1(x)\cdots\frac{1}{(m_r-1)!}\left(\frac{d}{dx}\right)^{m_r-1}h_r(x)Y(\iota(a),x)}
\ee
with $h_i(x)=\sum_{m\in\ZZ}h_i(m)x^{-m-1}$ where we have associated $h_i\otimes t^m=h_i(m)$. We also have
\be
Y(\iota(a),x)=E^{-}(-\overline{a},x)E^{+}(-\overline{a},x)ax^{\overline{a}}\ee
where 
\be 
E^{\pm}(\a,x)=\text{exp}\left(\sum_{m\in\pm\ZZ_{+}}\frac{\a(m)}{m}x^{-m}\right)\in(\text{End } V_L)[[x,x^{-1}]]\ee
for $\a\in\h$.
By Theorem 6.5.3 of \cite{LL}, $(V_L,Y,\mathbb{1},\omega)$ is a simple vertex algebra of central charge $d=\text{rank } L=\text{dim} \h$ with $\mathbb{1}=1\otimes 1$ and 
\be
\omega=\frac{1}{2}\sum_{i=1}^{d}h_i(-1)h_i(-1)\mathbb{1}\ee
where $\{h_1,\cdots,h_d\}$ forms an orthonormal basis of $\h$.
We also extend the automorphism $\hnu$ of $\hat{L}$ to an automorphism (which we also denote by $\hnu$) of $V_L$. Using the tensor product construction of $V_L$ (\ref{VL}) we take this extension to be $\nu\otimes\hnu$. We refer the reader to Section 2 of \cite{CalLM4} (\cite{L1}) for the details showing that this indeed defines an automorphism of $V_L$. 

Our next objective is to construct a $\hnu$-twisted $V_L$-module in the spirit of \cite{CalLM4} and \cite{L1}. We skip many details an focus on how this construction specializes to our case. 

Set
\be
\h_{(m)}=\{h\in\h|\hnu h=(-1)^mh\}\ee
for $m\in\ZZ$, identify $\h_{(m)}=\h_{(m\text{ mod }2)}$, and write
\be
\h=\h_{(0)}\oplus\h_{(1)}.
\ee
Observe that 
\be\label{zeroproj}
\h_{(0)}=\{\a+\nu\a|\a\in\h\}\ee
and
\be
\h_{(1)}=\{\a-\nu\a|\a\in\h\}.\ee
Let $P_0$ and $P_1$ be the projections of $\h$ onto $\h_{(0)}$ and $\h_{(1)}$ respectively. For $h\in\h$ and $m\in\ZZ/2\ZZ$ we set $h_{(m)}=P_mh$. Thus, we may write
\be
h_{(m)}=\frac{1}{2}(h+(-1)^m \nu h).\ee

We form the $\nu$-twisted affine Lie algebra associated to $\h$:
\be
\hat{\h}[\nu]=\h_{(0)}\otimes\CC[t,t^{-1}]\oplus\h_{(1)}\otimes t^{\frac{1}{2}}\CC[t,t^{-1}]\oplus\CC k\ee
where $k$ is central and
\be
[\a\otimes t^r,\b\otimes t^s]=\left<\a,\b\right>r\delta_{r+s,0}k\ee
for $r,s\in\frac{1}{2}\ZZ$, $\a\in\h_{(2r)}$, and $\beta\in\h_{(2s)}$. We have that $\hat{\h}[\nu]$ is $\frac{1}{2}\ZZ$-graded by weights such that 
\be 
\wt(\a\otimes t^m)=-m\hspace{.5in}\text{and}\hspace{.5in}\wt(k)=0.\ee

We set
\be
N=(1-P_0)\h\cap L\ee
and observe that in our setting we have
\be
N=\coprod_{i=1}^{l}\ZZ(\a_i-\nu\a_i).\ee
 Using Proposition 6.2 of \cite{L1}, let $\CC_{\tau}$ denote the one dimensional $\hat{N}$-module $\CC$ with character $\tau$ and write
\be T=\CC_{\tau}.\ee

Now consider the induced $\hat{L}_{\nu}$-module 
\be
U_{T}=\CC[\hat{L}_{\nu}]\otimes_{\CC[\hat{N}]}T\cong \CC[L/N],\ee
which is graded by weights and on which $\hat{L}_{\nu}$, $\h_{(0)}$, and $x^h$ for $h\in\h_{(0)}$ all naturally act. Set
\be
V_{L}^{T}=S[\nu]\otimes U_T\cong S\left(\hat{\h}[\nu]^{-}\right)\otimes \CC[L/N],\ee
which is naturally acted upon by $\hat{L}_{\nu}$, $\hat{\h}_{\frac{1}{2}\ZZ}$, $\h_{(0)}$, and $x^h$ for $h\in\h$.

For each $\a\in\h$ and $m\in\frac{1}{2}\ZZ$ define the operators on $V_L^T$
\be
\a_{(2m)}\otimes t^m\mapsto \a^{\hnu}(m)\ee
and set
\be
\a^{\hnu}(x)=\sum_{m\in\frac{1}{2}\ZZ}\a^{\hnu}(m)x^{-m-1}.\ee
Of most importance will be the $\hnu$-twisted vertex operators acting on $V_L^T$ for each $e_{\a}\in\hat{L}$
\be
Y^{\hnu}(\iota(e_{\a}),x)2^{-\frac{\left<\a,\a\right>}{2}}E^{-}(-\a,x)E^{+}(-\a,x)e_{\alpha}x^{\alpha_{(0)}+\frac{\left<\alpha_{(0)},\alpha_{(0)}\right>}{2}-\frac{\left<\alpha,\alpha\right>}{2}}.\ee
as defined in \cite{L1}, where
\begin{equation}
E^{\pm}(-\alpha,x)=\text{exp}\left(\sum_{n\in\pm\frac{1}{2}\mathbb{Z}_{+}}\frac{-\alpha_{(2n)}(n)}{n}x^{-n}\right),
\end{equation}
for $\a\in\h$. For $m\in\frac{1}{2}$ and $\a\in L$ define the component operators $\xa{}{m}$ by 
\begin{equation}\label{VertexOperators}
Y^{\hnu}(\iota(e_{\alpha}),x)=\sum_{m\in\frac{1}{3}\mathbb{Z}}\xa{}{m}x^{-m-\frac{\left<\alpha,\alpha\right>}{2}}.
\end{equation}
Our construction to this point has been fairly general, apart from the restrictions given by (\ref{isomcond1}) and (\ref{isomcond}). From now on we will assume that $L$ is a root lattice of type $A_{2n-1}, D_n,$ or $E_6$. In the case that we have a root lattice of type $A$ or $D$ we make the standard choice of simple roots, which we now recall. In the case of $A_{2n-1}$, we let
\be
L=\ZZ\a_1\oplus\cdots\oplus\ZZ\a_{2n-1} \text{ with } \a_i=\varepsilon_i-\varepsilon_{i-1}\ee
where $\{\varepsilon_i\}$ is the standard basis of $\mathbb{R}^{2n}$. This corresponds to the following standard labeling of the Dynkin diagram.

\begin{center}\begin{tikzpicture}[start chain]
\dnode{1}
\dnode{2}
\dydots
\dnode{n-1}
\dnode{n}
\dnode{n+1}
\dydots
\dnode{2n-2}
\dnode{2n-1}
\end{tikzpicture}\end{center}
In this case, our isometry $\nu$ corresponds to the Dynkin diagram folding and is defined on the simple roots by 
\be\label{isomA}
\nu\a_i=\a_{2n-i} \text{ for } 1\leq i\leq 2n-1,\ee
and extended $\ZZ$-linearly. Observe that $\a_n$ is fixed under $\nu$. Moreover, the corresponding set of roots is given by
\be
\Delta=\{\varepsilon_{i}-\varepsilon_j|1\leq i\neq j\leq 2n\},\ee
where we choose the following subset of positive roots
\be
\Delta_{+}=\{\varepsilon_{i}-\varepsilon_j|1\leq i<j\leq 2n\}\ee
so that we have 
\be \Delta=\Delta_{+}\cup (-\Delta_{+}).\ee
For $1\leq i<j\leq 2n$ we may write $\pm(\varepsilon_i-\varepsilon_j)=\pm(\a_i+\a_{i+1}+\cdots+\a_{j-1})$ from which we have
\be
\nu(\pm(\varepsilon_i-\varepsilon_j))=\pm(\varepsilon_{2n-j+1}-\varepsilon_{2n-i+1}).\ee
In the case of $D_n$, we have 
\be\begin{aligned}
L=\ZZ\a_1\oplus\cdots\oplus\ZZ\a_{n} &\text{ with } \a_i=\varepsilon_i-\varepsilon_{i+1} \text{ for } 1\leq i<n\\&\text{ and }\a_n=\varepsilon_{n-1}+\varepsilon_{n},\end{aligned}\ee
where $\{\varepsilon_i\}$ is the standard basis of $\mathbb{R}^{n}$. This corresponds to the following standard labeling of the Dynkin diagram

\begin{center}
\begin{tikzpicture}
\begin{scope}[start chain]
\dnode{1}
\dnode{2}
\node[chj,draw=none] {\dots};
\dnode{n-2}
\dnode{n-1}
\end{scope}
\begin{scope}[start chain=br going above]
\chainin(chain-4);
\dnodebr{n}
\end{scope}
\end{tikzpicture}\end{center}

In this case, $\nu$ is defined by 
\be\begin{aligned}
&\nu\a_i=\a_i \text{ for } 1\leq i\leq n-2\\
&\nu \a_{n-1}=\a_n\\
&\nu\a_n=\a_{n-1},\end{aligned}\ee
and extended $\ZZ$-linearly. Moreover, the corresponding set of roots is given by
\be
\Delta=\{\varepsilon_{i}\pm\varepsilon_j|1\leq i\neq j\leq n\},\ee
where we choose the following subset of positive roots
\be
\Delta_{+}=\{\varepsilon_{i}\pm\varepsilon_j|1\leq i<j\leq n\}\ee
so that we have 
\be \Delta=\Delta_{+}\cup (-\Delta_{+}).\ee
We follow a procedure similar to that above to determine that the action of $\nu$ on any root is
\be
\nu(\pm(\varepsilon_i\pm\varepsilon_j))=\pm(\varepsilon_i\pm(-1)^{\delta_{j,n}}\varepsilon_j),
\ee
for $1\leq i<j\leq n$.

In the case of $E_6$, we first consider the standard basis for the root system with simple roots given by
\be\begin{aligned}
\a_1&=\EE_1-\EE_2\\
\a_2&=\EE_2-\EE_3\\
\a_3&=\EE_3-\EE_4\\
\a_4&=\EE_4+\EE_5\\
\a_5&=-\frac{1}{2}(\EE_1+\EE_2+\EE_3+\EE_4+\EE_5)+\frac{\sqrt{3}}{2}\EE_6\\
\a_6&=\EE_4-\EE_5,\end{aligned}\ee
where $\{\EE_1,\dots,\EE_6\}$ is the standard basis of $\mathbb{R}^6$.
 We now consider the invertible change of basis matrix 
 \be 
 P=\begin{pmatrix}\frac{1}{3} & \frac{1}{3} & \frac{1}{3} & 0 & 0 & \frac{1}{\sqrt{3}} \\ -\frac{1}{6} & -\frac{1}{6}  & -\frac{1}{6}  & \frac{1}{2}  &-\frac{1}{2}   & -\frac{1}{2\sqrt{3}}  \\ -\frac{1}{6} & -\frac{1}{6}  & -\frac{1}{6}  & -\frac{1}{2}  & \frac{1}{2}   & -\frac{1}{2\sqrt{3}} \\ -\frac{1}{6} & -\frac{1}{6}  & -\frac{1}{6}  & \frac{1}{2}  & \frac{1}{2}   & \frac{1}{2\sqrt{3}}  \\ -\frac{1}{6} & -\frac{1}{6}  & -\frac{1}{6}  & -\frac{1}{2}  & -\frac{1}{2}   & \frac{1}{2\sqrt{3}}\\ \frac{1}{3} & \frac{1}{3} & \frac{1}{3} & 0 & 0 & -\frac{1}{\sqrt{3}}\\ \frac{1}{3} & \frac{1}{3} & -\frac{2}{3} & 0 & 0 & 0 \\ \frac{1}{3} & -\frac{2}{3} & \frac{1}{3} & 0 & 0 & 0 \\ -\frac{2}{3} & \frac{1}{3} & \frac{1}{3} & 0 & 0 & 0\end{pmatrix}
 \ee
 and observe that for all $\mathbf{v},\mathbf{w}\in\Delta_{E_6}$, where $\Delta_{E_6}$ is the set of roots of $E_6$, we have 
 \be
 \left<P\mathbf{v},P\mathbf{w}\right>= \left<\mathbf{v},\mathbf{w}\right>,\ee
 where on the left we take the standard inner product on $\mathbb{R}^9$ and on the right the standard inner product on $\mathbb{R}^6$. The result is that  a copy of the $E_6$ root system may be embedded into $\mathbb{R}^9$. Now we make the identification $\mathbb{R}^9=\mathbb{R}^3\times\mathbb{R}^3\times \mathbb{R}^3$ and observe that result is a root system made up of nine dimensional vectors of the form $(\mathbf{v}_1,\mathbf{v}_2,\mathbf{v}_3)\in \mathbb{R}^3\times\mathbb{R}^3\times \mathbb{R}^3$. The positive roots are of two main types, either one of $\mathbf{v}_k=\EE_{i}-\EE_j$ with $1\leq i< j\leq3$ while the other two are the zero vector, or one of the 27 combinations where each $\mathbf{v}_i$ is taken from $\{-\frac{2}{3}\EE_1+\frac{1}{3}\EE_2+\frac{1}{3}\EE_3, \frac{1}{3}\EE_1-\frac{2}{3}\EE_2+\frac{1}{3}\EE_3,\frac{1}{3}\EE_1+\frac{1}{3}\EE_2-\frac{2}{3}\EE_3\}$.
With this identification we relabel the simple roots to be  
\be\begin{aligned}\label{syme6basis}
\a_1&=(\mathbf{0},\mathbf{0},\EE_2-\EE_3)\\
\a_2&=(\mathbf{0},\mathbf{0},\EE_1-\EE_2)\\
\a_3&=(\frac{1}{3}\EE_1-\frac{2}{3}\EE_2+\frac{1}{3}\EE_3,-\frac{2}{3}\EE_1+\frac{1}{3}\EE_2+\frac{1}{3}\EE_3,-\frac{2}{3}\EE_1+\frac{1}{3}\EE_2+\frac{1}{3}\EE_3)\\
\a_4&=(\mathbf{0},\EE_1-\EE_2,\mathbf{0})\\
\a_5&=(\mathbf{0},\EE_2-\EE_3,\mathbf{0})\\
\a_6&=(\EE_2-\EE_3,\mathbf{0},\mathbf{0}).
\end{aligned}\ee

This corresponds to the following labeling of the Dynkin diagram.
\begin{center}\begin{tikzpicture}
\begin{scope}[start chain]
\foreach \dyni in {1,...,5} {
\dnode{\dyni}
}
\end{scope}
\begin{scope}[start chain=br going above]
\chainin (chain-3);
\dnodebr{6}
\end{scope}
\end{tikzpicture}\end{center}

Here we have
\be\begin{aligned}
\nu\a_1&=\a_5\\
\nu\a_2&=\a_4\\
\nu\a_3&=\a_3\\
\nu\a_6&=\a_6.\end{aligned}\ee

 Using this choice for the root vectors the action of $\nu$ is given by
 \be\label{nue6}
 \nu(\mathbf{v}_1,\mathbf{v}_2,\mathbf{v}_3)=(\mathbf{v}_1,\mathbf{v}_3,\mathbf{v}_2).
\ee
In each of these cases, we take $\left<\cdot,\cdot\right>:L\times L\to \ZZ$ to be the bilinear form given by 
\be
\left<\a,\b\right>=2\frac{(\b,\a)}{(\b,\b)},\ee
where $(\cdot,\cdot)$ is the standard Euclidean inner product given by $(\varepsilon_i,\varepsilon_j)=\delta_{i,j}$. Recall from (\ref{liftingdef}), in each of these cases the isometry $\nu$ lifts to an automorphism $\hnu$ of $V_L$ where 
\be
\hnu e_{\a_i}=e_{\nu \a_i}.\ee

Now we focus on the construction of the Lie algebras of type $A_{2n-1}$, $D_{n}$, and $E_6$, as well as their affine twisted counterparts from the appropriate root lattice and isometry $\nu$. Each of these cases will be handled simultaneously and we step back to point out significant differences when needed. Consider the vector space defined by 
\be
\g=\h\oplus\coprod_{\a\in\Delta}x_{\a},\ee
where $\h$ is defined as in (\ref{cartan}) and $\{x_{\a}\}$ is a set of symbols, where $\Delta$ is the set of roots corresponding to $L$. 

 We endow $\g$ with the structure of a Lie algebra via the bracket defined 
\be
[h,x_{\a}]=\left<h,\a\right>x_{\a},~~[\h,\h]=0\ee
where $h\in\h$ and $\a\in\Delta$ and
\begin{equation}
   [x_{\alpha},x_{\beta}] = \left\{
     \begin{array}{lr}
       \epsilon_{C_0}(\alpha,-\alpha)\alpha & \text{ if }  \alpha+\beta=0\\
        \epsilon_{C_0}(\alpha,\beta)x_{\alpha+\beta} & \text{ if }  \alpha+\beta\in\Delta\\
        0   & \text{ otherwise}
     \end{array}
   \right.
\end{equation}  
Observe that $\g$ is a Lie algebra isomorphic to one of type $A_{2n-1}$, $D_n$, or $E_6$ depending on the choice of $L$. We extend the bilinear form $\left<\cdot,\cdot\right>$ to $\g$ by
\be
\left<h,x_{\a}\right>=\left<x_{\a},h\right>=0\ee
and
\begin{equation}
\left<x_{\alpha},x_{\beta}\right>=\left\{\begin{array}{lrr}
\epsilon_{C_0}(\alpha,-\alpha)     &\text{if}~~~~\alpha+\beta=&0 \\
0     &\text{if}~~~~~\alpha+\beta\neq&0
\end{array}\right.\end{equation}

Following \cite{L1} and \cite{CalLM4}, we use our extension of $\nu:L\to L$ to $\hnu:\hat{L}\to\hat{L}$ to lift the automorphism $\nu:\h\to\h$ to an automorphism $\hnu:\g\to\g$ by setting
\be
\hnu x_{\a}=\psi (\alpha) x_{\nu\a}
\ee
for all $\a\in\Delta$, where $\psi$ is defined in (\ref{psilift}). 
Here, we are using our particular choices of $\hnu$ (extended to $\mathbb{C}\{L \}$) and our section $e$.

For $m\in\mathbb{Z}$ set
\begin{equation}
\g_{(m)}=\{x\in\g|\hnu(x)=(-1)^m x\}.
\end{equation}
 Now we decompose $\g$ into eigenspaces of $\hnu$ as follows
\be
\g_{(0)}=\coprod_{i=1}^{l}\CC(\a_i+\nu\a_i)\oplus\coprod_{\a\in\Delta}\CC(x_{\a}+x_{\nu\a})\ee
and
\be
\g_{(1)}=\coprod_{i=1}^{l}\CC(\a_i-\nu\a_i)\oplus\coprod_{\a\in\Delta}\CC(x_{\a}-x_{\nu\a})\ee
Form the $\hnu$-twisted affine Lie algebra associated to $\g$ and $\hnu$:
\begin{equation}\label{defineg}
\hat{\g}[\hnu]=\coprod_{n\in\mathbb{Z}}\g_{(n)}\otimes t^{n/2}\oplus\mathbb{C}k=\coprod_{n\in\frac{1}{2}\mathbb{Z}}\g_{(2n)}\otimes t^{n}\oplus\mathbb{C}k
\end{equation}
with
\begin{equation}
[x\otimes t^m,y\otimes t^n]=[x,y]\otimes t^{m+n}+\left<x,y\right>m\delta_{m+n,0}k
\end{equation}
and
\begin{equation}
[k,\hat{\g}[\hnu]]=0,
\end{equation}
for $m,n\in\frac{1}{2}\mathbb{Z}$, $x\in\g_{(2m)}$, and $y\in\g_{(2n)}$
Set 
\begin{equation}
\tilde{\g}[\hnu]=\hat{\g}[\hnu]\oplus\mathbb{C}d,
\end{equation}
with
\begin{equation}
[d,x\otimes t^n]=n x\otimes t^n,
\end{equation}
for $x\in\g_{(2n)}$, $n\in\frac{1}{2}\mathbb{Z}$ and $[d,k]=0$. The Lie algebra $\tilde{g}[\hnu]$ is isomorphic to $A_{2n-1}^{(2)}$, $D_n^{(2)}$, or $E_6^{(2)}$ depending on the choice of $L$, and is $\frac{1}{2}\mathbb{Z}$-graded.

We now recall the following result that gives $V_L^T$ the structure of a $\hat{\g}[\hnu]$-module:
\begin{thm}(Theorem 3.1 \cite{CalLM4}, Theorem 9.1 \cite{L1}, Theorem 3 \cite{FLM2})\label{reptheorem}
The representation of $\hat{\h}[\nu]$ on $V_L^T$ extends uniquely to a Lie algebra representation of $\hat{\g}[\hnu]$ on $V_L^T$ such that
\begin{equation*}
(x_{\alpha})_{(2n)}\otimes t^n\mapsto \xa{}{n}
\end{equation*}
for all $m\in\frac{1}{2}\mathbb{Z}$ and $\alpha\in L$. Moreover $V_L^T$ is irreducible as a $\hat{\g}[\hnu]$-module.
\end{thm}

As in Section 2 of \cite{CalLM4} (also Section 6 of \cite{L1}) we have a tensor product grading on $V_L^T$ given by the action of $L^{\hnu}(0)$ with
\begin{equation}
\begin{aligned}
L^{\hnu}(0)1&=\frac{1}{16}\sum_{j=1}^{1}j(2-j)\text{dim }\h_{(j)}1\\
&=\frac{1}{16}\left(\text{dim }\h_{(1)}\right),
\end{aligned}
\end{equation}
Where
\be Y^{\hnu}(\omega,x)=\sum_{m\in\ZZ}L^{\nu}(m)x^{-n-2}.\ee
We will write $\wt(1)=\frac{\text{dim }\h_{(1)}}{16}$. For a homogeneous element $v\in V_L^T$ we have
\begin{equation}\label{wt}
\hat{L}^{\hnu}(0)v=\left(\wt (v) +\frac{\text{dim }\h_{(1)}}{16}\right)v.
\end{equation} 
For each of our special cases we have
\be\begin{aligned}
\wt(1)&=\frac{n-1}{16}, \text{ for } A_{2n-1}\\
\wt(1)&=\frac{1}{16}, \text{ for } D_n\\
\wt(1)&=\frac{1}{8}, \text{ for } E_6,\end{aligned}\ee
which carries over to the weight of any homogeneous element in (\ref{wt}).

By Proposition 6.3 of \cite{DL1} we have that
\begin{multline}\label{YomegaCommutator}
[Y^{\hnu}(\omega,x_1),Y^{\hnu}(\iota(e_\alpha),x_2)]\\
 = x_2^{-1}\frac{d}{dx_2}Y^{\hnu}(\iota(e_\alpha),x_2)\delta(x_1/x_2) - \frac{1}{2}\langle \alpha,\alpha \rangle x_2^{-1} Y^{\hnu}(\iota(e_\alpha),x_2)\frac{d}{dx_1}\delta(x_1/x_2),
\end{multline}
where $\delta(x) = \sum_{n \in \mathbb{Z}}x^n$. Taking $\mathrm{Res}_{x_2} \mathrm{Res}_{x_1}$ of $x_1 x_2^m$ times of both sides of (\ref{YomegaCommutator}) immediately gives
\begin{equation}
[L^{\hnu}(0),\xa{}{m}]=\left(-m-1+\frac{1}{2}\left<\alpha,\alpha\right>\right)\xa{}{m}
\end{equation}
and thus 
\begin{equation}
\wt(\xa{}{m})=-m-1+\frac{1}{2}\left<\alpha,\alpha\right>.
\end{equation}

Let 
\be\label{dualbasis}\mathcal{B^{\star}}=\{\lambda_{i}\in\mathbb{Q}\otimes_{\ZZ}L|1\leq i\leq l \text{ and } \left<\lambda_i,\a_j\right>=\delta_{i,j}\}\ee be the basis of the lattice dual to $L$ with respect to the basis $\mathcal{B}=\{\a_i|1\leq i\leq l\}$. The $\lambda_i$ are in the rational span of $\mathcal{B}$ and so using (\ref{zeroproj}) we have
\be (\lambda_i)_{(0)}=\frac{1}{2}(\lambda_i+\nu\lambda_i).\ee
We now use the eigenvalues of the operators $(\lambda_i)_{(0)}(0)$ to construct a new charge grading as follows. Let 
\be\label{charge1}
\mathfrak{o}_i=\{\lambda\in\mathcal{B}^{\star}|\lambda=\nu^{m}\lambda_i \text{ for some }m\in\mathbb{N}\},
\ee
that is, $\mathfrak{o}_i$ is the orbit of $\lambda_i$ under the action of $\nu$. Observe that the distinct $\mathfrak{o}_i$ form a partition of $\mathcal{B}^{\star}$. Now define
\be\label{charge2}
\lambda^{(i)}=\sum_{\lambda\in\mathfrak{o}_i}(\lambda)_{(0)}.\ee
Finally, we define our charge grading to be a tuple built from the eigenvalues of $\lambda^{(i)}(0)$ as $i$ runs over all distinct orbits $\mathfrak{o}_i$. That is 
\be
\text{ch}(\xa{}{m})=\left(\left<\a,\lambda^{(i_1)}\right>,\left<\a,\lambda^{(i_2)}\right>,\dots,\left<\a,\lambda^{(i_k)}\right>\right),\ee
where the sets $\mathfrak{o}_{i_j}$ are all distinct (and disjoint), with $i_1<i_2<\cdots <i_k$.  It is easy to see that 
\be
\text{ch}(\xa{}{m})=\left(d_{i_1}\left<\a,(\lambda_{i_1})_{(0)}\right>,d_{i_2}\left<\a,(\lambda_{i_2})_{(0)}\right>,d_{i_k}\left<\a,(\lambda_{i_k})_{(0)}\right> \right),
\ee
where $d_{i_j}=|\mathfrak{o}_{i_j}|$. Further if $\text{ch}(v)=(m_1,\dots, m_k)$ then we say the total charge of $v$ is $m_1+m_2+\cdots+m_k$. 

In this case of $A_{2n-1}$ the amounts to the $n$-tuple given by 
\be 
\text{ch}(\xa{}{m})=\left<2\left<\a,(\lambda_{1})_{(0)}\right>,\dots,2\left<\a,(\lambda_{n-1})_{(0)}\right>,\left<\a,(\lambda_{n})_{(0)}\right>\right).\ee
In the case of $D_n$ we have
\be\text{ch}(\xa{}{m})=\left<\left<\a,(\lambda_{1})_{(0)}\right>,\dots,\left<\a,(\lambda_{n-2})_{(0)}\right>,2\left<\a,(\lambda_{n-1})_{(0)}\right>\right).\ee
Finally in the $E_6$ case we have 
\be
\text{ch}(\xa{}{m})=\left(2\left<\a,(\lambda_{1})_{(0)}\right>,2\left<\a,(\lambda_{2})_{(0)}\right>,\left<\a,(\lambda_{3})_{(0)}\right>,\left<\a,(\lambda_{6})_{(0)}\right>\right).\ee

Consider the following $\nu$-stable Lie subalgebra of $\g$:
\begin{equation}
\n=\coprod_{\alpha\in\Delta_{+}}\mathbb{C}x_\alpha.\end{equation}\color{black}
and its $\hnu$-twisted affinization
\begin{equation}
\hat{\n}[\hnu]=\coprod_{m\in\mathbb{Z}}\n_{(m)}\otimes t^{m/2}\oplus\mathbb{C}k=\coprod_{m\in\frac{1}{2}\mathbb{Z}}\n_{(2m)}\otimes t^{m}\oplus\mathbb{C}k,
\end{equation}
where $\n_{(m)}$ is the $(-1)^m$-eigenspace of $\hnu$ in $\n$. As in \cite{CalLM4}, \cite{CalMPe}, and \cite{PS}, we take
\begin{equation}\label{nilrad}\begin{aligned}
\overline{\n}[\hnu]&=\coprod_{m\in\frac{1}{2}\mathbb{Z}}\n_{(2m)}\otimes t^m\\
\overline{\n}[\hnu]_{+}&=\coprod_{m\in\frac{1}{2}\mathbb{Z}_{\geq 0}}\n_{(2m)}\otimes t^m\\
\overline{\n}[\hnu]_{-}&=\coprod_{m\in\frac{1}{2}\mathbb{Z}_{<0}}\n_{(2m)}\otimes t^m
\end{aligned}\end{equation}

Now we will investigate the properties of operators from $\hat{\n}[\hnu]$ acting on $V_L^T$ by investigating the vertex operators (\ref{VertexOperators}) and their modes. Using equation (2.93) of \cite{CalLM4}, we have that
\begin{equation}\label{limit}
Y^{\hnu}(\hnu v,x)=\lim_{x^{\frac{1}{2}}\to(-1)x^{\frac{1}{2}}}Y^{\hnu}(v,x).
\end{equation}
In our setting ($ADE$), we have two cases to consider: when $\a$ is fixed under $\nu$ or when $\a$ is not fixed under $\nu$, with $\nu^{2}\a=\a$. In the case that $\a$ is fixed we have
 \be
 Y^{\hnu}(\iota(e_{\a}),x)=\lim_{x^{\frac{1}{2}}\to(-1)x^{\frac{1}{2}}}Y^{\hnu}(\iota(e_{\a}),x),
\ee
and thus
 \be
 Y^{\hnu}(\iota(e_{\a}),x)\in (\text{End }V_L^T)[[x,x^{-1}]].\ee
 In terms of component operators this implies that 
 \be\label{zero}\xa{}{m}=0 \text{ for all } m\in\frac{1}{2}+\ZZ.\ee
  In the other case we have that 
 \be
 Y^{\hnu}(\iota(e_{\nu\a}),x)=\lim_{x^{\frac{1}{2}}\to(-1)x^{\frac{1}{2}}}Y^{\hnu}(\iota(e_{\a}),x),
\ee
and thus
\be\label{nuoperator}\begin{aligned}
x_{\nu\a}^{\hnu}(m)&=\xa{}{m} \text{ for } m\in\ZZ\\
x_{\nu\a}^{\hnu}(m)&=-\xa{}{m} \text{ for } m\in\frac{1}{2}+\ZZ.\end{aligned}\ee

In our setting, the commutator formula among twisted vertex operators becomes
\begin{equation}
[Y^{\hnu}(u,x_1),Y^{\hnu}(v,x_2)] = \frac{1}{2}x_2^{-1}\mathrm{Res}_{x_0} \bigg(\sum_{j \in \mathbb{Z}/2\mathbb{Z}} \delta \bigg( (-1)^j\frac{(x_1 - x_0)^{1/2}}{x_2^{1/2}}\bigg) Y^{\hnu}(Y(\hnu^j u,x_0)v,x_2) \bigg).
\end{equation}
As a consequence, using the fact that $\left<\alpha,\beta\right>\geq -1$ for all $\alpha,\beta\in\Delta_{+}$ we have the following commutator formula for twisted vertex operators
\begin{equation}
\left[Y^{\hnu}(\iota(e_{\alpha}),x_1),Y^{\hnu}(\iota(e_{\beta}),x_2)\right]=x_2^{-1}\frac{1}{2}\sum_{j\in\mathbb{Z}/2\mathbb{Z}}\delta\left((-1)^{j}\frac{x_1^{\frac{1}{2}}}{x_2^{\frac{1}{2}}}\right)Y^{\hnu}\left(x_{\nu^j\alpha}(0)\iota(e_{\beta}),x_2\right),
\end{equation} 
leading to the following lemma:

\begin{lem}\label{comlemma}
For $\left<\a,\b\right>\in\Delta^{+}$ we have
\begin{equation}
\left[Y^{\hnu}(\iota(e_{\alpha}),x_1),Y^{\hnu}(\iota(e_{\beta}),x_2)\right]=0 \text{ when } \left<\nu^j\a,\b\right>\geq 0 \text{ for } 0\leq j\leq 1
\end{equation}
\begin{equation}\begin{aligned}\label{com1}
\left[Y^{\hnu}(\iota(e_{\alpha}),x_1),Y^{\hnu}(\iota(e_{\beta}),x_2)\right]&=x_2^{-1}\frac{1}{2}\epsilon_{C_0}(\alpha,\beta)\delta\left(\frac{x_1^{\frac{1}{2}}}{x_2^{\frac{1}{2}}}\right)Y^{\hnu}(\iota(e_{\alpha+\beta}),x_2) \\
&\text{ when } \left<\alpha,\b\right>=-1 \text{ and }\left<\nu\a,\b\right>\geq 0
\end{aligned}\end{equation}
\begin{equation}\begin{aligned}
\left[Y^{\hnu}(\iota(e_{\alpha}),x_1),Y^{\hnu}(\iota(e_{\beta}),x_2)\right]=x_2^{-1}\frac{1}{2}\sum_{j=0}^{1}\epsilon_{C_0}(\nu^j\a,\b)\delta\left((-1)^{j}\frac{x_1^{\frac{1}{2}}}{x_2^{\frac{1}{2}}}\right)Y^{\hnu}\left(\iota(e_{\nu^j\a+\beta}),x_2\right)\\
\text{ when } \left<\a,\b\right>=\left<\nu\a,\b\right>=-1.
\end{aligned}\end{equation}  
\end{lem}
\begin{proof}
This follows from the standard vertex algebra fact that $x_{\alpha}\left(-\left<\a,\b\right>-1\right)e_{\b}=\epsilon_{C_0}(\a,\b)e_{\a+\b}$.
\end{proof}
We can use Lemma \ref{comlemma} to make the following statement regarding component operators.

\begin{cor}\label{comcor}
For $\a,\b\in \Delta_{+}$ and $r,s\in\frac{1}{2}\ZZ$ chosen from a suitable subset so that the operators are nonzero (see (\ref{zero})), we have
\be\label{comcor2}
[\xa{}{r},\xb{s}]=0 \text{ when } \left<\a,\b\right>\geq 0 \text{ and } \left<\nu\a,\b\right>\geq 0,\ee
\be\label{comcor2'}
[\xa{}{r},\xb{s}]=\frac{1}{2} \epsilon_{C_0}(\alpha,\beta)x^{\hnu}_{\a+\b}(r+s) \text{ when } \left<\a,\b\right>=-1 \text{ and }  \left<\nu\a,\b\right>\geq 0,\ee
and
\be\label{comcor3}
[\xa{}{r},\xb{s}]= \frac{1}{2} \epsilon_{C_0}(\a,\b)x^{\hnu}_{\a+\b}(r+s) + (-1)^{2r}\frac{1}{2} \epsilon_{C_0}(\nu\a,\b)x^{\hnu}_{\nu\a+\b}(r+s) 
 \text{ when } \left<\a,\b\right>=-1 \text{ and }  \left<\nu\a,\b\right>=-1.\ee
\end{cor}
We now present two technical Lemmas that will be of use in the simplification of (\ref{comcor3}).

\begin{lem}\label{techlemma1}
Suppose that $\a,\b\in\Delta_{+}$ such that $\left<\a,\b\right>=\left<\nu\a,\b\right>=-1$. Then either $\nu\a=\a$ or $\nu\b=\b$.
\end{lem}
\begin{proof}
We separate our argument into cases depending on type. First considering the $A_{2n-1}$ case. 

Suppose $\a,\b\in\Delta_+$ are such that $\left<\a,\b\right>=\left<\nu\a,\b\right>=-1$. Since we are dealing with positive roots of $A_{2n-1}$ we may write 
\be
\a=\EE_i-\EE_j \text{ with }1\leq i<j\leq 2n
\ee
and
\be
\b=\EE_k-\EE_l \text{ with } 1\leq k<l\leq 2n.\ee
The isometry $\nu$ acts as follows on the roots written in this form
\be
\nu\a=\EE_{2n+1-j}-\EE_{2n+1-i}.\ee
Observe that we have 
\be\label{innerproductexpansion}
-1=\left<\a,\b\right>=\delta_{i,k}+\delta_{j,l}-\delta_{j,k}-\delta_{i,l}.\ee
We claim that only one of $\delta_{j,k},\delta_{i,l}$ is nonzero. Indeed, otherwise we have 
\be 
\b=\EE_{k}-\EE_{l}=\EE_{j}-\EE_{i}=\a\notin\Delta_{+}.\ee
Thus we have $i\neq k$, $j\neq l$, and either $j=k$ or $i=l$ (but not both). Without loss of generality we take $j=k$, and thus $i\neq l$. Now considering an expansion of $\left<\nu\a,\b\right>$ similar to (\ref{innerproductexpansion}) we have 
\be\label{typeone}
2n+1-i=j
\ee
or
\be\label{typetwo}
2n+1-j=l.
\ee
If (\ref{typeone}) holds we have $\a=\EE_{i}-\EE_{2n+1-i}$ which is fixed under $\nu$, while if (\ref{typetwo}) holds we have $\b=\EE_j-\EE_{2n+1-j}$ with is also fixed under $\nu$.

Now we consider the $D_n$ case, here we may write positive roots in the form 
\be 
\a=\EE_{i}+(-1)^a\EE_{j}\text{ with } 1\leq i<j\leq n\ee
and
\be
\b=\EE_{k}+(-1)^b\EE_{l}\text{ with }1\leq k<l\leq n,\ee
where $a,b\in\{0,1\}$. In this setting we have 
\be
\nu\a=\EE_{i}+(-1)^{a+\delta_{j,n}}\EE_j.\ee
We may also write 
\be
-1=\left<\a,\b\right>=\delta_{i,k}+(-1)^a\delta_{j,k}+(-1)^b\delta_{i,l}+(-1)^{a+b}\delta_{j,l}\ee
and
\be
-1=\left<\nu\a,\b\right>=\delta_{i,k}+(-1)^{a+\delta_{j,n}}\delta_{j,k}+(-1)^b\delta_{i,l}+(-1)^{a+b+\delta_{j,n}}\delta_{j,l}.\ee
Using this to expand $\left<\a,\b\right>-\left<\nu\a,\b\right>=0$ we have
\be
\left((-1)^{a}-(-1)^{a+\delta_{j,n}}\right)\delta_{j,k}-\left((-1)^{a+b}-(-1)^{a+b+\delta_{j,n}}\right)\delta_{j,l}=0.\ee
We now claim that we cannot have both $j=k$ and $j=l$, otherwise we would have $\b=2\epsilon_{j}\notin\Delta_{+}$. This leaves us with the following
\begin{itemize}
\item[Case 1:] $j\neq k$ and $j\neq l$ and thus at least one of $j$ or $l$ is not $n$ and the corresponding root is fixed under $\nu$.
\item[Case 2:] $j=k$ and $j\neq l$ and thus $(-1)^a=(-1)^{a+\delta_{j,n}}$ which implies $j\neq n$ and $\a$ is fixed.
\item[Case 3:] $j\neq k$ and $j=l$ and thus $(-1)^{a+b}=(-1)^{a+b+\delta_{j,n}}$ which implies $j\neq n$ and $\a$ is fixed.
\end{itemize}

Now we consider the $E_6$ case. Here we use a symmetric version of the root system as described above. We take the root system to be made up of nine dimensional vectors which we consider of the form $(\mathbf{v}_1,\mathbf{v}_2,\mathbf{v}_3)\in \mathbb{R}^3\times\mathbb{R}^3\times \mathbb{R}^3$. The positive roots are of two main types, either one of $\mathbf{v}_i=\EE_{i}-\EE_j$ with $1\leq i< j\leq3$ while the other two are the zero vector, or one of the 27 combinations where each $\mathbf{v}_i$ is taken from $\{-\frac{2}{3}\EE_1+\frac{1}{3}\EE_2+\frac{1}{3}\EE_3, \frac{1}{3}\EE_1-\frac{2}{3}\EE_2+\frac{1}{3}\EE_3,\frac{1}{3}\EE_1+\frac{1}{3}\EE_2-\frac{2}{3}\EE_3\}$.
In this setting we take the roots to be
\be\begin{aligned}
\a_1&=(\mathbf{0},\mathbf{0},\EE_2-\EE_3)\\
\a_2&=(\mathbf{0},\mathbf{0},\EE_1-\EE_2)\\
\a_3&=(-\frac{1}{3}\EE_1+\frac{2}{3}\EE_2-\frac{1}{3}\EE_3,-\frac{2}{3}\EE_1+\frac{1}{3}\EE_2+\frac{1}{3}\EE_3,-\frac{2}{3}\EE_1+\frac{1}{3}\EE_2+\frac{1}{3}\EE_3)\\
\a_4&=(\mathbf{0},\EE_1-\EE_2,\mathbf{0})\\
\a_5&=(\mathbf{0},\EE_2-\EE_3,\mathbf{0})\\
\a_6&=(\EE_2-\EE_3,\mathbf{0},\mathbf{0}).
\end{aligned}\ee
We have that
\be \nu(\mathbf{v}_1,\mathbf{v}_2,\mathbf{v}_3)=(\mathbf{v}_1,\mathbf{v}_3,\mathbf{v}_2).\ee
Finally, we let $\a=(\mathbf{v}_1,\mathbf{v}_2,\mathbf{v}_3)$ and $\b=(\mathbf{w}_1,\mathbf{w}_2,\mathbf{w}_3)$ and expand the inner product 
\be
\left<\a,\b\right>=\left<\mathbf{v}_1,\mathbf{w}_1\right>+\left<\mathbf{v}_2,\mathbf{w}_2\right>+\left<\mathbf{v}_3,\mathbf{w}_3\right>\ee
and notice that the equation $\left<\a,\b\right>-\left<\nu\a,\b\right>=0$ implies that 
\be\label{lastcondition}
\left<\mathbf{v}_2-\mathbf{v}_3,\mathbf{w}_2-\mathbf{w}_3\right>=0.\ee
Now we make the observation that
\be \mathbf{v}_2-\mathbf{v}_3,\mathbf{w}_2-\mathbf{w}_3\in\{\EE_i-\EE_j|1\leq i,j\leq 3\}.\ee

It is clear that all choices of $\mathbf{v}_i$ and $\mathbf{w}_j$ that make $\a,\b\in\Delta_+$, none of them allow for (\ref{lastcondition}) to hold unless 
$\mathbf{v}_2-\mathbf{v}_3=0$ (and thus $\a$ is fixed) or $\mathbf{w}_2-\mathbf{w}_3=0$ (and thus $\b$ is fixed).
\end{proof}

In view of Lemma \ref{techlemma1}, if $\left<\a,\b\right>=-1 \text{ and }  \left<\nu\a,\b\right>=-1$, then we can write
(\ref{comcor3}) as
\be
[\xa{}{r},\xb{s}]=\left(\frac{1}{2} \epsilon_{C_0}(\a,\b) + (-1)^{2r}\frac{1}{2} \epsilon_{C_0}(\a,\b)\right)x^{\hnu}_{\a+\b}(r+s) = \epsilon_{C_0}(\a,\b)x^{\hnu}_{\a+\b}(r+s) 
\ee
if $\nu \a = \a$ (which also implies $r \in \ZZ$). Finally, if $\nu \a \neq \a$ (in which case $\nu \b = \b$ and $s \in \ZZ$) we have
\begin{eqnarray*}
[\xa{}{r},\xb{s}]&=&-[\xb{s}, \xa{}{r}]\\
&=& -\epsilon_{C_0}(\beta, \alpha) x^{\hnu}_{\a+\b}(r+s). 
\end{eqnarray*}

\section{Principal subspaces}
Following \cite{CalLM4} and \cite{CalMPe} (cf. \cite{CalLM1}-\cite{CalLM3}, \cite{FS1} - \cite{FS2}, and other works involving principal subspaces) we define the principal subspace of a highest weight $\tilde{\g}[\hnu]$-module:
\begin{defn}For any standard $\tilde{\g}[\hnu]$-module $V$ with highest weight vector $v$, define its principal subspace $W$ to be 
\begin{equation*}
W=U\left(\overline{\n}[\hnu]\right)\cdot v.
\end{equation*}
\end{defn}

Denote by $W_L^T$ the principal subspace of the standard $\tilde{\g}[\hnu]$-module $V_L^T$,
\begin{equation}
W_L^T=U\left(\overline{\n}[\hnu]\right)\cdot v_{\Lambda},
\end{equation}
where $\Lambda\in\left(\h_{(0)}\oplus\mathbb{C}k\oplus\mathbb{C}d\right)^{*}$ is the fundamental weight of $\tilde{\g}[\hnu]$ defined by $\left<\Lambda,k\right>=1$, $\left<\Lambda,\h_{(0)}\right>=0$, and $\left<\Lambda,d\right>=0$, and $v_{\Lambda}$ is a highest weight vector of $V_L^T$. We have
\begin{equation}
W_L^T=U\left(\overline{\n}[\hnu]_{-}\right)\cdot v_{\Lambda}.
\end{equation}
We take $v_{\Lambda}=1=1\otimes 1\in V_L^T (\cong S[\nu]\otimes U_T)$.

Consider the surjection
\begin{equation}
\begin{aligned}
F_{\Lambda}:U\left(\hat{\g}[\hnu]\right)&\to V_L^T\\
a&\mapsto a\cdot \vL
\end{aligned}\end{equation}
and its restriction
\begin{equation}\begin{aligned}
f_{\Lambda}:U\left(\overline{\n}[\hnu]\right)&\to W_L^T\\
a&\mapsto a\cdot \vL
\end{aligned}\end{equation}
As in previous work (\cite{CalLM1}-\cite{CalLM4}, \cite{CalMPe},and others), our main goal is to describe the kernel of $f_{\Lambda}$, giving us a presentation of $W_L^T$.

Following \cite{CalLM4}(cf. \cite{CalLM1}-\cite{CalLM3}) we set
\begin{equation}
N_L^T=U\left(\hat{\g}[\hnu]\right)\otimes_{U\left(\hat{\g}[\hnu]\right)_{\geq 0}}\mathbb{C}\vL^N,
\end{equation}
i.e. the generalized Verma module where,
\begin{equation}
\hat{\g}[\hnu]_{\geq 0}=\coprod_{n\geq 0}\g_{(n)}\otimes t^{n/2}\oplus\mathbb{C}k.
\end{equation}
We also use the principal subspaces of the generalized Verma module
\begin{equation}
W_L^{T,N}=U\left(\overline{\n}[\hnu]\right)\cdot \vL^N\subset N_L^T.
\end{equation}
Now we define the following related surjections
\begin{equation}
\begin{aligned}
F^N_{\Lambda}:U\left(\hat{\g}[\hnu]\right)&\to N_L^T\\
a&\mapsto a\cdot \vL^N\\
f^N_{\Lambda}:U\left(\overline{\n}[\hnu]\right)&\to W_L^{T,N}\\
a&\mapsto a\cdot \vL^N
\end{aligned}\end{equation}
as well as
\begin{equation}\begin{aligned}
\Pi:N_L^T&\to V_L^T\\
a\cdot\vL^N&\mapsto a\cdot\vL\\
\pi:W_L^{T,N}&\to W_L^T\\
a\cdot\vL^N&\mapsto a\cdot\vL.
\end{aligned}\end{equation}

We now look at some results involving the vertex operators $Y^{\hnu}(e^{\a},x)$. These will be the source of our description of $\text{Ker}(f_{\Lambda})$. Using the fact
\begin{equation}
E^{+}(\alpha,x_1)E^{-}(\beta,x_2)=\prod_{p\in\mathbb{Z}/2\mathbb{Z}}\left(1-(-1)^p\frac{x_2^{\frac{1}{2}}}{x_1^{\frac{1}{2}}}\right)^{\left<\nu^p\alpha,\beta\right>}E^{-}(\beta,x_2)E^{+}(\alpha,x_1),\end{equation}
 (cf. \cite{L1}, \cite{CalLM4}) we have the following results:

\begin{prop}\label{relationgens} For $\alpha,\beta\in\Delta^{+}$ we have
\begin{align}
&Y^{\hnu}(\iota(e_{\a}),x)Y^{\hnu}(\iota(e_{\b}),x)=0, \text{ when } \left<\nu^j\a,\b\right>\geq 0 \text{ for } 0\leq j\leq 2\\
&\label{pt2}\lim_{x_2^{\frac{1}{2}}\to x_1^{\frac{1}{2}}}\left(x_1^{\frac{1}{2}}-x_2^{\frac{1}{2}}\right)Y^{\hnu}(\iota(e_{\a}),x_1)Y^{\hnu}(\iota(e_{\b}),x_2)=0, \text{ when } \left<\alpha,\beta\right>=-1 \text{ and } \left<\nu\alpha,\beta\right>\geq 0 \\
&\lim_{x_2^{\frac{1}{2}}\to x_1^{\frac{1}{2}}}\left(x_1^{\frac{1}{2}}+x_2^{\frac{1}{2}}\right)Y^{\hnu}(\iota(e_{\a}),x_1)Y^{\hnu}(\iota(e_{\b}),x_2)=0, \text{ when } \left<\alpha,\beta\right>\geq 0 \text{ and } \left<\nu\alpha,\beta\right>=-1\\
&\lim_{x_2^{\frac{1}{2}}\to x_1^{\frac{1}{2}}}\left(x_1-x_2\right)Y^{\hnu}(\iota(e_{\a}),x_1)Y^{\hnu}(\iota(e_{\b}),x_2)=0, \text{ when } \left<\nu^j\alpha,\beta\right>=-1 \text{ for } 0\leq j\leq 1.
\end{align}
\end{prop}

By taking appropriate coefficients of the formal variables in the expressions from Proposition \ref{relationgens} we have the following sums which, when applied to any vector $v\in V_L^T$, yield zero.

\begin{cor} We have
\begin{equation}\label{infinite1}
R(\alpha,\beta|t)=\sum_{\substack{m_1,m_2\in\frac{1}{2}\ZZ \\ m_1+m_2=-t}}\xa{}{m_1}\xbet{m_2}  \text{ when } \left<\nu^j\a,\b\right>\geq 0 \text{ for } j=0,1 ,
\end{equation}

\begin{equation}\label{infinite2}\begin{aligned}
R(\alpha,\beta|t)=\sum_{\substack{m_1,m_2\in\frac{1}{2}\ZZ \\ m_1+m_2+\frac{1}{2}=-t}}&\left(\xa{}{m_1+\frac{1}{2}}\xbet{m_2}- \xa{}{m_1}\xbet{m_2+\frac{1}{2}}\right)\\
& \text{ when } \left<\alpha,\beta\right>=-1 \text{ and } \left<\nu\alpha,\beta\right>\geq 0 ,
\end{aligned}\end{equation}

\begin{equation}\label{infinite3}\begin{aligned}
R(\alpha,\beta|t)=\sum_{\substack{m_1,m_2\in\frac{1}{2}\ZZ \\ m_1+m_2+\frac{1}{2}=-t}}&\left(\xa{}{m_1+\frac{1}{2}}\xbet{m_2}+ \xa{}{m_1}\xbet{m_2+\frac{1}{2}}\right)\\
& \text{ when } \left<\alpha,\beta\right>\geq 0 \text{ and } \left<\nu\alpha,\beta\right>=-1 ,
\end{aligned}\end{equation}

and

\begin{equation}\label{infinite4}\begin{aligned}
R(\alpha,\beta|t)=\sum_{\substack{m_1,m_2\in\frac{1}{2}\ZZ \\ m_1+m_2+1=-t}}&\left(\xa{}{m_1+1}\xbet{m_2}- \xa{}{m_1}\xbet{m_2+1}\right)\\
&\text{ when } \left<\nu^j\alpha,\beta\right>=-1 \text{ for } j=0,1,
\end{aligned}\end{equation}
when applied to any $v\in V_L^T$ yield zero.
\end{cor}

We define the following truncated finite sums

\begin{equation}\label{finite1}
R^0(\alpha,\beta|t)=\sum_{\substack{m_1,m_2\in\frac{1}{2}\ZZ_{<0} \\ m_1+m_2=-t}}\xa{}{m_1}\xbet{m_2}  \text{ when } \left<\nu^j\a,\b\right>\geq 0 \text{ for } j=0,1,
\end{equation}

\begin{equation}\label{finite2}\begin{aligned}
R^0(\alpha,\beta|t)=\sum_{\substack{m_1,m_2\in\frac{1}{2}\ZZ_{<0} \\ m_1+m_2+\frac{1}{2}=-t}}&\left(\xa{}{m_1+\frac{1}{2}}\xbet{m_2}- \xa{}{m_1}\xbet{m_2+\frac{1}{2}}\right)\\
& \text{ when } \left<\alpha,\beta\right>=-1 \text{ and } \left<\nu\alpha,\beta\right>\geq 0,
\end{aligned}\end{equation}

\begin{equation}\label{finite3}\begin{aligned}
R^0(\alpha,\beta|t)=\sum_{\substack{m_1,m_2\in\frac{1}{2}\ZZ_{<0} \\ m_1+m_2+\frac{1}{2}=-t}}&\left(\xa{}{m_1+\frac{1}{2}}\xbet{m_2}+ \xa{}{m_1}\xbet{m_2+\frac{1}{2}}\right)\\
& \text{ when } \left<\nu \alpha,\beta\right> \geq 0 \text{ and } \left<\nu\alpha,\beta\right>=-1,
\end{aligned}\end{equation}

and

\begin{equation}\label{finite4}\begin{aligned}
R^0(\alpha,\beta|t)=\sum_{\substack{m_1,m_2\in\frac{1}{2}\ZZ_{<0} \\ m_1+m_2+1=-t}}&\left(\xa{}{m_1+1}\xbet{m_2}- \xa{}{m_1}\xbet{m_2+1}\right)\\
&\text{ when } \left<\nu^j\alpha,\beta\right>=-1 \text{ for } j=0,1,
\end{aligned}\end{equation}

For the following proposition, we require a certain completion of $\un\overline{\mathfrak{n}}[\hnu]_{+}$, which we call $\widetilde{\un\overline{\mathfrak{n}}[\hnu]_{+}}$. We refer the reader to \cite{LW} and \cite{S2} for details of the construction of this completion.

\begin{prop} For $\a,\b\in\Delta^{+}$ and $R(\a,\b|t)$ described above we may write
\be
R(\a,\b|t)=R^0(\a,\b|t)+a
\ee
where $R^0(\a,\b|t)$ is as above and $a\in\widetilde{\un\overline{\mathfrak{n}}[\hnu]_{+}}$.
\end{prop}

\begin{proof}

It is clear we can rewrite the infinite sums (\ref{infinite1}) in this form because in this case the operators $\xa{}{m}$ and $\xbet{n}$ commute, so we any term that contains a positive mode can easily be written as an element of $\widetilde{\un\overline{\mathfrak{n}}[\hnu]_{+}}$.

We now focus on expressions of the form (\ref{infinite2}); the arguments involving (\ref{infinite3}) and (\ref{infinite4}) will be similar. We first decompose
\begin{align*}
R(\alpha,\b|t)=&\sum_{m_1\geq 0, m_2 < 0}\left(\xa{}{m_1+\frac{1}{2}}\xbet{m_2}- \xa{}{m_1}\xbet{m_2+\frac{1}{2}}\right)\\
+&R^0(\a,\b|t)+\sum_{m_1 < 0, m_2\geq 0}\left(\xa{}{m_1+\frac{1}{2}}\xbet{m_2}- \xa{}{m_1}\xbet{m_2+\frac{1}{2}}\right).
\end{align*}
The third summand is clearly a member of $\widetilde{\un\overline{\mathfrak{n}}[\hnu]_{+}}$ so we will focus on the first. Using Corollary \ref{comcor} we write

\begin{align*}
&\sum_{m_1\geq 0, m_2 < 0}\left(\xa{}{m_1+\frac{1}{2}}\xbet{m_2}- \xa{}{m_1}\xbet{m_2+\frac{1}{2}}\right)\\
&=\sum_{m_1\geq0, m_2 < 0}\left(\xbet{m_2}\xa{}{m_1+\frac{1}{2}}-\frac{1}{2}\epsilon_{C_0}(\a,\b)x^{\hnu}_{\a+\b}\left(m_1+m_2+\frac{1}{2}\right)\right.\\
&\hspace{.35in}\left.-\xbet{m_2+\frac{1}{2}}\xa{}{m_1}+\frac{1}{2}\epsilon_{C_0}(\a,\b)x^{\hnu}_{\a+\b}\left(m_1+m_2+\frac{1}{2}\right)\right)\\
&=\sum_{m_1\geq0, m_2 < 0}\left(\xbet{m_2}\xa{}{m_1+\frac{1}{2}}-\xbet{m_2+\frac{1}{2}}\xa{}{m_1}\right)\in \widetilde{\un\overline{\mathfrak{n}}[\hnu]_{+}}
\end{align*}

\end{proof}

Define the following left ideals of $\un$:

\be
J=\sum_{\substack{t\in\ZZ_{>0}\\, \nu \alpha_i = \alpha_i}}\un R^0(\a_i,\a_i|t)+
\sum_{\substack{t\in\frac{1}{2}\ZZ_{>0}\\  \nu \alpha_i \neq \alpha_i}}\un R^0(\a_i,\a_i|t)
\ee
as well as
\be
I_{\Lambda}=J+\un\overline{\mathfrak{n}}[\hnu]_{+}.
\ee
For $A_{2n-1}$, we let $i=1,\dots,n$. For $D_n$, we let $i=1,\dots,n-1$, and for $E_6$ we let $i=1,2,3,6$.

\begin{rem}
The ideal $I_\Lambda$ is defined analogously to the ideals found in \cite{CalLM1}-\cite{CalLM4}, \cite{CalMPe}, and \cite{PS}. In the following sections, we investigate the properties of $I_\Lambda$, and, as in \cite{CalLM1}-\cite{CalLM4}, \cite{CalMPe}, and \cite{PS}, show that it gives the desired presentation of the principal subspace. 
\end{rem}

\section{Important Mappings}

Following \cite{CalLM4}, \cite{CalMPe}, and \cite{PS}, we make use of certain shift automorphisms of $\un$. For $\gamma\in\h_{(0)}$ and a character of the root lattice $\theta:L\to\CC$ consider the linear map
\begin{equation}\begin{aligned}
\ta{}:\overline{\mathfrak{n}}[\hnu]&\to\overline{\mathfrak{n}}[\hnu]\\
\xa{}{m}&\mapsto\theta_{}(\alpha)\xa{}{m+\left<\a_{(0)},\gamma\right>}.
\end{aligned}\end{equation}
For suitably chosen $\gamma$ and $\theta$, this is a Lie algebra automorphism that extends to an automorphism of $\un$ which we also denote by $\ta{}$ defined by
\begin{equation}\begin{aligned}
\ta{}&\left(\xbbet{i_r}{m_r}\cdots\xbbet{i_1}{m_1}\right)\\
&=\theta(\b_{i_1}+\cdots+\b_{i_r})\xbbet{i_r}{m_r+\left<(\beta_{i_r})_{(0)},\gamma\right>}\cdots\xbbet{i_1}{m_1+\left<(\beta_{i_1})_{(0)},\gamma\right>},
\end{aligned}\end{equation}
where $\b_{i_j}\in\Delta^{+}$.  As in \cite{CalLM4}, \cite{CalMPe}, and \cite{PS}, we use automorphisms associated to the elements of $\h_{(0)}$ associated to the duals of the simple roots
 \begin{equation}
 \gamma_i=(\lambda_{i})_{(0)} = \frac{1}{2}\left( \lambda_i + \nu \lambda_i \right)
 \end{equation}
where, for $A_{2n-1}$, we let $i=1,\dots,n$, for $D_n$, we let $i=1,\dots,n-1$, and for $E_6$ we let $i=1,2,3,6$.
 as well as the characters $\theta_i$ defined by
 
 \be \theta_i(\a_j)=(-1)^{\left<\lambda^{(i)},\a_j\right>}.\ee
 In particular, if $L$ is of type $A_{2n-1}$, we have:

 \begin{equation}
\theta_i(\alpha_i)=-1
\end{equation}
 \begin{equation}
\theta_i(\alpha_j)=1 \text{ if } i\neq j \text{ and } j \in \{1,\dots ,2n-1\}
\end{equation}
for $i \in \{1,\dots , n-1 \}$ 
and we have
 \begin{equation}
\theta_n(\alpha_j)=1 \text{ if } j \in \{1,\dots ,2n-1\} \\
\end{equation}
If $L$ is of type $D_{n}$, we have:
 \begin{equation}
\theta_{n-1}(\alpha_{n-1})=-1
\end{equation}
 \begin{equation}
\theta_{n-1}(\alpha_j)=1 \text{ if } j\in \{1,2,\dots,n-2,n \}
\end{equation}
and define
 \begin{equation}
\theta_i(\alpha_j)=1 \text{ if } j \in \{1,\dots ,n\}
\end{equation}
for $i \in \{1,\dots , n-1 \}$.\\
If $L$ is of type $E_{6}$, we have: \begin{equation}
\theta_i(\alpha_i)=-1
\end{equation}
 \begin{equation}
\theta_i(\alpha_j)=1 \text{ if } i\neq j \text{ and } j \in \{1,\dots, 6\}
\end{equation}
for $i \in \{1,2\}$ 
and we have
 \begin{equation}
\theta_i(\alpha_j)=1 \text{ if } j \in \{1,\dots ,2n-1\} \text{ and } i \in \{3,6\}. \\
\end{equation}
The maps $\ta{i}$ have inverses given by $\tau_{-\gamma_i,\theta_i^{-1}}$.
The following Lemma follows immediately from the action of the maps $\ta{i}$:
\begin{lem}\label{lempt0}
We have that 
\begin{equation}
\ta{i}\left(\un \on_{+} + \un\xa{i}{-1}\right) \subset \un \on_{+} 
\end{equation}
if $\nu \alpha_i = \alpha_i$ and
\begin{equation}
\ta{i}\left(\un \on_{+} + \un\xa{i}{-\frac{1}{2}}\right) \subset \un \on_{+} 
\end{equation}
if $\nu \alpha_i \neq \alpha_i$.
\end{lem} 
\begin{proof}
This immediately follows from the fact that $\langle \gamma_i, \beta_{(0)} \rangle \geq 0$ for each $\beta \in \Delta_+$, and the fact that 
\begin{equation}
\ta{i}\left({\xa{i}{-1}}\right) = \theta_i(\alpha_i)\xa{i}{0} \in \un \on_+
\end{equation}
if $\nu \alpha_i = \alpha_i$ and
\begin{equation}
\ta{i}\left({\xa{i}{-\frac{1}{2}}}\right) = \theta_i(\alpha_i)\xa{i}{0} \in \un \on_+
\end{equation}
if $\nu \alpha_i \neq \alpha_i$.
\end{proof}

We now present three Lemmas that are essential in determining Ker$f_\Lambda$. The lemmas and their proofs largely mirror similar lemmas in \cite{CalMPe} and \cite{PS}.
 \begin{lem}\label{Containment}We have
 \be\label{lempt1}
 \ta{i}\left(I_{\Lambda}+\un\xa{i}{-1}\right)= I_{\Lambda}
 \ee
 if $ \nu \a_i = \a_i$ and
 \be
 \ta{i}\left(I_{\Lambda}+\un\xa{i}{-\frac{1}{2}}\right)= I_{\Lambda}
 \ee
 if $ \nu \a_i \neq \a_i$.
 
 \begin{proof}
We first show that
 \be\label{cont1}
 \ta{i}\left(I_{\Lambda}+\un\xa{i}{-1}\right)\subset I_{\Lambda},
 \ee
 if $ \nu \a_i = \a_i$ and
 \be\label{cont2}
 \ta{i}\left(I_{\Lambda}+\un\xa{i}{-\frac{1}{2}}\right)\subset I_{\Lambda}.
 \ee
 if $ \nu \a_i \neq \a_i$
 In view of Lemma \ref{lempt0}, it is sufficient to show that 
 $\ta{i} (R^0(\a_j,\a_j|t)) \in I_\Lambda$.
 If $i \neq j$, we have that 
 \begin{equation}
 \ta{i} (R^0(\a_j,\a_j|t)) = R^0(\a_j,\a_j|t) \in I_\Lambda.
 \end{equation}
 If $\nu \alpha_i = \alpha_i$, we have that
 \begin{equation*}\begin{aligned}
\ta{i}(R^0(\a_i,\a_i|t))&=\ta{i}\left(\sum_{\substack{m_1,m_2\in\ZZ_{<0} \\ m_1+m_2=-t}}\xa{i}{m_1}\xa{i}{m_2}\right)\\
&=\theta_i(2\a_i)\left(\sum_{\substack{m_1,m_2\in\ZZ_{<0} \\ m_1+m_2=-t}}\xa{i}{m_1+1}\xa{i}{m_2+1}\right)\\
&=\theta_i(2\a_i)\left(R^0(\a_i,\a_i|t-2)+2\xa{1}{-t+2}\xa{i}{0}\right)\\
&\in J+\un\on_{+}=I_{\Lambda}.\end{aligned}\end{equation*}
If $\nu \alpha_i \neq \alpha_i$, we have
\be\begin{aligned}
 \ta{i}\left(R^0(\a_i,\a_i|t)\right)&=\ta{i}\left(\sum_{\substack{m_1,m_2\in\frac{1}{2}\ZZ{<0} \\ m_1+m_2=-t}}\xa{i}{m_1}\xa{i}{m_2}\right)\\
 &=\theta_i(2\a_i)\left(\sum_{\substack{m_1,m_2\in\frac{1}{2}\ZZ{<0} \\ m_1+m_2=-t}}\xa{i}{m_1+\frac{1}{2}}\xa{2}{m_2+\frac{1}{2}}\right)\\
 &=\theta_{i}(2\a_i)\left(R^0\left(\a_i,\a_i|t-1\right)+2\xa{i}{-t+1}\xa{i}{0}\right).\\
 &\in I_{\Lambda}
 \end{aligned}
 \ee

Since $J$ is the left ideal generated by the $R^0(\a_i,\a_i|t)$, we have that $\ta{i}(J)\subset I_{\Lambda},$ establishing (\ref{cont1}) and (\ref{cont2}). 

For the other containment, we will prove equivalent statements
\be
\tau_{-\gamma_i,\theta_i^{-1}}(I_{\Lambda})\subset I_{\Lambda}+\un\xa{i}{-1}.
\ee
if $\nu \alpha_i = \alpha_i$
and
\be
\tau_{-\gamma_i,\theta_i^{-1}}(I_{\Lambda})\subset I_{\Lambda}+\un\xa{i}{-\frac{1}{2}}.
\ee
if $\nu \alpha_i \neq \alpha_i$

Using
\be \label{invmap1}
\tau_{-\gamma_i,\theta_i^{-1}}(\xa{j}{m})=\theta_1^{-1}(\a_i)\xa{j}{m-\delta_{i,j}}
\ee
if $\nu \a_i = \a_i$ and
\be \label{invmap2}
\tau_{-\gamma_i,\theta_i^{-1}}(\xa{j}{m})=\theta_1^{-1}(\a_i)\xa{j}{m-\frac{1}{2}\delta_{i,j}}
\ee
if $\nu \a_i \neq \a_i$,
it is clear that 
\be
\tau_{-\gamma_i,\theta_i^{-1}}(\un\on_{+})\subset \un\on_{+}+\un\xa{i}{-1}
\ee
if $\nu \a_i = \a_i$
and
\be
\tau_{-\gamma_i,\theta_i^{-1}}(\un\on_{+})\subset \un\on_{+}+\un\xa{i}{-\frac{1}{2}}
\ee
if $\nu \a_i \neq \a_i$. Furthermore, if $i \neq j$ we have that
\be
\tau_{-\gamma_i,\theta_i^{-1}}(R(\a_j,\a_j|t)) = R(\a_j,\a_j|t) \in I_\Lambda.
\ee
Now we investigate the action of $\tau_{-\gamma_i,\theta_i^{-1}}$ on $R^0(\a_i,\a_i|t)$.
If $\nu \a_i = \a_i$, we have
\be\begin{aligned}
\tau_{-\gamma_i,\theta_i^{-1}}(R^0(\a_i,\a_i|t))&=\tau_{-\gamma_i,\theta_i^{-1}}\left(\sum_{\substack{m_1,m_2\in\ZZ_{<0} \\ m_1+m_2=-t}}\xa{i}{m_1}\xa{i}{m_2}\right)\\
&=\theta^{-1}(2\a_i)\left(\sum_{\substack{m_1,m_2\in\ZZ_{<0} \\ m_1+m_2=-t}}\xa{i}{m_1-1}\xa{i}{m_2-1}\right)\\
&=\theta^{-1}(2\a_i)\left(R^0(\a_i,\a_i|t+2)-2\xa{i}{-t-1}\xa{i}{-1}\right)\\
&\in I_{\Lambda}+\un\xa{i}{-1}.\end{aligned}\ee
If  $\nu \a_i \neq \a_i$, we have
\be\begin{aligned}
\tau_{-\gamma_i,\theta_i^{-1}}(R^0(\a_i,\a_i|t))&=\tau_{-\gamma_i,\theta_i^{-1}}\left(\sum_{\substack{m_1,m_2\in\frac{1}{2}\ZZ_{<0} \\ m_1+m_2=-t}}\xa{i}{m_1}\xa{i}{m_2}\right)\\
&=\theta^{-1}(2\a_i)\left(\sum_{\substack{m_1,m_2\in \frac{1}{2}\ZZ_{<0} \\ m_1+m_2=-t}}\xa{i}{m_1-\frac{1}{2}}\xa{i}{m_2-\frac{1}{2}}\right)\\
&=\theta^{-1}(2\a_i)\left(R^0(\a_i,\a_i|t+1)-2\xa{i}{-t-\frac{1}{2}}\xa{i}{-\frac{1}{2}}\right)\\
&\in I_{\Lambda}+\un\xa{i}{-\frac{1}{2}}.\end{aligned}\ee
Since $J$ is the left ideal generated by the $R(\a_i,\a_i|t)$, we have that 
\be
\tau_{-\gamma_i,\theta_i^{-1}}(J)\subset I_{\Lambda}+\un\xa{i}{-1}.
\ee
if $\nu \alpha_i = \alpha_i$
and
\be
\tau_{-\gamma_i,\theta_i^{-1}}(J)\subset I_{\Lambda}+\un\xa{i}{-\frac{1}{2}}.
\ee
if $\nu \alpha_i \neq \alpha_i$,
thus completing the proof of the reverse containment.
 \end{proof}
  \end{lem}
 
 As in \cite{CalLM4}, \cite{CalMPe}, and \cite{PS} define linear maps $\ps{i}$ for $1\leq i\leq l$ on $\un$ by:
\begin{equation}\begin{aligned}
\ps{i}:\un&\to\un\\
a&\mapsto \tau_{-\alpha_i,\theta_i^{-1}}(a)\xa{i}{-1},
\end{aligned}\end{equation}
if $\nu\a_i=\a_i$ and
\begin{equation}\begin{aligned}
\ps{i}:\un&\to\un\\
a&\mapsto \tau_{-\alpha_i,\theta_i^{-1}}(a)\xa{i}{-\frac{1}{2}},
\end{aligned}\end{equation}
if $\nu\a_i\neq \a_i$.

In the spirit of Lemma \ref{techlemma1}, we present another technical lemma that will be of use for the next result. 

\begin{lem}\label{techlemma2}Suppose $\a_i$ is a simple root and $\a\in\Delta_+$ such that $\left<\a_i,\a\right>=2$. Then $\a=\a_i$.
\end{lem}

\begin{proof}

We begin with the $A_{2n-1}$ setting, where $\a_i=\EE_{i}-\EE_{i+1}$ and $\a=\EE_k-\EE_l$ for $1\leq i\leq 2n$ and $1\leq k<l\leq 2n+1$. We now observe that 
\be\label{typeA}
\left<\a_i,\a\right>=\left<\EE_i-\EE_{i+1},\EE_k-\EE_l\right>=\delta_{i,k}+\delta_{i+1,l}-\delta_{i+1,k}-\delta_{i,l}=2.\ee
This implies that $i=k$ and $i+1=l$ and thus $\a=\a_i$.

Moving on to the $D_n$ setting we will explicitly check the cases that are not covered by (\ref{typeA}). We begin with $\a_i=\EE_i-\EE_{i+1}$ for $1\leq i\leq n-1$ and $\a=\EE_k+\EE_l$ for $1\leq k<l\leq n$. We see that 
\be
\left<\a_i,\a\right>=\left<\EE_i-\EE_{i+1},\EE_k+\EE_l\right>=\delta_{i,k}+\delta_{i,l}-\delta_{i+1,k}-\delta_{i+1,l}\ee
which implies that $k=i=l$ and thus $\a=2\EE_k\notin\Delta_+$. The remaining $D_n$ cases are similar.

Finally we explore the $E_6$ setting, again using the symmetric basis as described in (\ref{syme6basis}). This setting separates into a four cases. 
\begin{itemize}
\item[Case 1:] $\a_i\neq \a_3$ and $\a=(\mathbf{v}_1,\mathbf{v}_2,\mathbf{v}_3)$ where one of the $\mathbf{v}_i$ is $\EE_{k}-\EE_{l}$ for $1\leq k<l\leq 3$. The calculations for this case are is similar to that of type $A_{2n-1}$ as described above andf will be omitted. 
\item[Case 2:] $\a_i\neq\a_3$ and $\a=(\mathbf{v}_1,\mathbf{v}_2,\mathbf{v}_3)$ is so that the $\mathbf{v}_i \in\{-\frac{2}{3}\EE_1+\frac{1}{3}\EE_2+\frac{1}{3}\EE_3, \frac{1}{3}\EE_1-\frac{2}{3}\EE_2+\frac{1}{3}\EE_3,\frac{1}{3}\EE_1+\frac{1}{3}\EE_2-\frac{2}{3}\EE_3\}$. It is clear that no combination of $\a_i$ and $\mathbf{v}_1,\mathbf{v}_2,\mathbf{v}_3$ will produce $\left<\a_i,\a\right>=2$ in this case.
\item[Case 3:] We have the simple root $\a_3$ and $\a=(\mathbf{v}_1,\mathbf{v}_2,\mathbf{v}_3)$ where one of the $\mathbf{v}_i$ is $\EE_{k}-\EE_{l}$ for $1\leq k<l\leq 3$. This case is essentially the same as case 2. 
\item[Case 4:] We have the simple root $\a_3$ and $\a=(\mathbf{v}_1,\mathbf{v}_2,\mathbf{v}_3)$ is so that the $\mathbf{v}_i \in\{-\frac{2}{3}\EE_1+\frac{1}{3}\EE_2+\frac{1}{3}\EE_3, \frac{1}{3}\EE_1-\frac{2}{3}\EE_2+\frac{1}{3}\EE_3,\frac{1}{3}\EE_1+\frac{1}{3}\EE_2-\frac{2}{3}\EE_3\}$. We first observe that for all vectors $\mathbf{u},\mathbf{v}\in\{-\frac{2}{3}\EE_1+\frac{1}{3}\EE_2+\frac{1}{3}\EE_3, \frac{1}{3}\EE_1-\frac{2}{3}\EE_2+\frac{1}{3}\EE_3,\frac{1}{3}\EE_1+\frac{1}{3}\EE_2-\frac{2}{3}\EE_3\}$ we have $\left<\mathbf{u},\mathbf{v}\right>=\frac{2}{3}$ if and only if $\mathbf{u}=\mathbf{v}$, otherwise $\left<\mathbf{u},\mathbf{v}\right>=-\frac{1}{3}$. The result follows immediately from this fact.
\end{itemize}

\end{proof}

The following lemmas will be useful in the next section when determining Ker$f_\Lambda$:

\begin{lem}\label{lempt01}
We have that 
\begin{equation}
\ps{i}\ta{i}\left(\un \on_{+} + \un\xa{i}{-1}\right) \subset I_{\Lambda} 
\end{equation}
if $\nu \alpha_i = \alpha_i$ and
\begin{equation}
\ps{i}\ta{i}\left(\un \on_{+} + \un\xa{i}{-\frac{1}{2}}\right) \subset I_{\Lambda}
\end{equation}
if $\nu \alpha_i \neq \alpha_i$.
\end{lem} 

\begin{proof}
We begin with the case that $\a_i$ is fixed under the action of $\nu$.  We begin by by examining $\ps{i}\ta{i}$ on the elements of $\on_+$. Assume that $m\geq 0$ and we now consider
 \be
 \ps{i}\ta{i}(\xa{}{m})=\xa{}{m+\left<\a_{(0)},\gamma_i\right>-\left<\a_{(0)},\a_i\right>}\xa{i}{-1},\ee
 where $\a\in\Delta_+$. In this case we have $\gamma_i=\lambda_i$ and make use of the isometry to write
\be 
\left<\a_{(0)},\gamma_i\right>-\left<\a_{(0)},\a_i\right>=\left<\a,\lambda_i\right>-\left<\a,\a_i\right>.
\ee
Observe that $\left<\a,\a_i\right>\in\{-1,0,1,2\}$. We use this fact to decompose this argument into a few cases. 

We begin with the case that $\left<\a,\a_i\right>=2$ which by Lemma \ref{techlemma2} we have that $\a=\a_i$ and thus $\left<\a,\lambda_i\right>=1$. Observe that 
\be
\ps{i}\ta{i}(\xa{i}{0})=\xa{i}{-1}\xa{i}{-1}=R^0(\a_i,\a_i|2)\in I_{\Lambda}.\ee
Now if $m\geq 1$ we have 
\be\begin{aligned}
\ps{i}\ta{i}(\xa{i}{m})&=\xa{i}{m-1}\xa{i}{-1}\\&=\xa{i}{-1}\xa{i}{m-1}\in\un\on_+\subset I_{\Lambda}.\end{aligned}\ee

We now move on to the case when $\left<\a,\a_i\right>=1$ and thus $\left<\a,\lambda_i\right>\geq 1$. In this case, we have that $\left<\a,\lambda_i\right>= 1$ or $\left<\a,\lambda_i\right>= 2$, and so
\be \begin{aligned}
\ps{i} \ta{i} (\xa{}{m}) = \xa{}{m-1+\left<\a,\lambda_i\right>} \xa{i}{-1} = \xa{i}{-1}\xa{}{m-1+\left<\a,\lambda_i\right>} \in \un\on_+ \subset I_{\Lambda}.
\end{aligned} \ee

When $\left<\a,\a_i\right>=0$, we have that $\left<\a,\lambda_i\right>\geq 0$. In this case, we have
\be \begin{aligned}
\ps{i} \ta{i} (\xa{}{m}) = \xa{}{m+\left<\a,\lambda_i\right>} \xa{i}{-1} = \xa{i}{-1}\xa{}{m + \left<\a,\lambda_i\right>} \in \un\on_+ \subset I_{\Lambda}.
\end{aligned} \ee

Finally, we examine the case when $\left<\a,\a_i\right>=-1$. In this case, we have that $\left<\a,\lambda_i\right> = 0$ or $\left<\a,\lambda_i\right>=1 $. In this case, we have
\be \begin{aligned}
\ps{i} \ta{i} (\xa{}{m}) &= \xa{}{m+1+\left<\a,\lambda_i\right>} \xa{i}{-1} \\
&=  \xa{i}{-1}\xa{}{m+1+\left<\a,\lambda_i\right>}+c\xaa{}{i}{m + \left<\a,\lambda_i\right>} \in \un\on_+ \subset I_{\Lambda},
\end{aligned} \ee
where $c$ is a nonzero constant given by Corollary \ref{comcor}.
Finally, we have that 
\begin{eqnarray*}
\ps{i}\ta{i}(\xa{i}{-1}) & = & \xa{i}{-2}\xa{i}{-1}\\
& = & 2 R^0(\a_i,\a_i|3),
\end{eqnarray*}
so that $\ps{i}\ta{i}(\un \xa{i}{-1}) \subset I_\Lambda$.

Now, we examine the case where $\a_i$ is not fixed by $\nu$. Throughout our calculations for this case we make use of the fact that by use of the isometry $\nu$ we may write
\be\label{innerexp}
\left<\a_{(0)},\gamma_i\right>-\left<\a_{(0)},\a_i\right>=\frac{1}{2}\left(\left<\a,\lambda_i\right>-\left<\a,\a_i\right>\right)+\frac{1}{2}\left(\left<\nu\a,\lambda_i\right>-\left<\nu\a,\a_i\right>\right)
\ee
Following our previous argument we use the fact $\left<\a,\a_i\right>,\left<\nu\a,\a_i\right>\in\{-1,0,1,2\}$ to decompose into a few subcases. Those subcases without explicit calculations follow by use of the isometry $\nu$.

First, if $\left<\a,\a_i\right>=2$ by Lemma \ref{techlemma2} we have $\a=\a_i$ and thus $\left<\a,\lambda_i\right>=1$, $\left<\nu\a,\a_i\right>=0$,  and $\left<\nu\a,\lambda_i\right>=0$. So we have 
\be
\ps{i}\ta{i}(\xa{i}{0})=\xa{i}{\frac{1}{2}}\xa{i}{\frac{1}{2}}=R^0(\a_i,\a_i|1)\in J\subset I_{\Lambda}.
\ee
Further if $m\geq\frac{1}{2}$ we see that 
\be\begin{aligned}
\ps{i}\ta{i}(\xa{i}{m})&=\xa{i}{m+\frac{1}{2}}\xa{i}{\frac{1}{2}}\\
&=\xa{i}{\frac{1}{2}}\xa{i}{m+\frac{1}{2}}\in \un\on_+ \subset I_{\Lambda}\end{aligned}\ee

Now we consider the case when $\left<\a,\a_i\right>,\left<\nu\a,\a_i\right>\in\{0,1\}$ thus $\left<\a,\lambda_i\right>\geq\left<\a,\a_i\right>$ and $\left<\nu\a,\lambda_i\right>\geq\left<\nu\a,\a_i\right>$. So using (\ref{innerexp}) we have
\be
\left<\a_{(0)},\gamma_i\right>-\left<\a_{(0)},\a_i\right>\geq 0
\ee 
and thus for all $m\geq 0$
\be\begin{aligned}
\ps{i}\ta{i}(\xa{}{m})&=\xa{}{m+\left<\a_{(0)},\gamma_i\right>-\left<\a_{(0)},\a_i\right>}\xa{i}{\frac{1}{2}}\\
&=\xa{i}{\frac{1}{2}}\xa{}{m+\left<\a_{(0)},\gamma_i\right>-\left<\a_{(0)},\a_i\right>}\in \un\on_+ \subset I_{\Lambda}\end{aligned}\ee

Now we move on to the subcase where $\left<\a,\a_i\right>\in\{0,1\}$  (thus $\left<\a,\lambda_i\right>\geq\left<\a,\a_i\right>$) and $\left<\nu\a,\a_i\right>=-1$ (thus $\left<\nu\a,\lambda_i\right>\in\{0,1\}$). Observe that in this setting we may write 
\be
\left<\a_{(0)},\gamma_i\right>-\left<\a_{(0)},\a_i\right>=\frac{1}{2}+a
\ee 
with $a\in\frac{1}{2}\ZZ_{\geq 0}$. Using Corollary \ref{comcor} we may write, for $m\geq 0$, 
\be\begin{aligned}
\ps{i}\ta{i}(\xa{}{m})&=\xa{}{m+\frac{1}{2}+a}\xa{i}{-\frac{1}{2}}\\
&=\xa{i}{-\frac{1}{2}}\xa{}{m+\frac{1}{2}+a}+cx^{\hnu}_{\a_i+\nu\a}\left(m+a\right)\in \un\on_+ \subset I_{\Lambda},\end{aligned}\ee
where $c$ is a nonzero constant. 

After combining cases by use of the isometry $\nu$ we are left with the final case of $\left<\a,\a_i\right>=\left<\nu\a,\a_i\right>=-1$, which by Lemma \ref{techlemma1} we have $\nu\a=\a$ and thus $\left<\a,\lambda_i\right>=\left<\nu\a,\lambda_i\right>\geq 0$. This implies 
\be
\left<\a_{(0)},\gamma_i\right>-\left<\a_{(0)},\a_i\right>=1+a
\ee 
where $a\in\ZZ_{\geq 0}$. Finally using Corollary \ref{comcor}, for $m\geq 0$ we have 
\be\begin{aligned}
\ps{i}\ta{i}(\xa{}{m})&=\xa{}{m+1+a}\xa{i}{-\frac{1}{2}}\\
&=\xa{i}{-\frac{1}{2}}\xa{}{m+1+a}+cx^{\hnu}_{\a_i+\nu\a}\left(m+\frac{1}{2}+a\right)\in \un\on_+ \subset I_{\Lambda},\end{aligned}\ee
where $c$ is a nonzero constant. Finally, we have that 
\begin{eqnarray*}
\ps{i}\ta{i}(\xa{i}{-\frac{1}{2}}) & = & \xa{i}{-1}\xa{i}{-\frac{1}{2}}\\
& = & 2 R^0(\a_i,\a_i|\frac{3}{2}),
\end{eqnarray*}
so that $\ps{i}\ta{i}(\un \xa{i}{-\frac{1}{2}}) \subset I_\Lambda$.
\end{proof}

\begin{lem} \label{CompContainment}
 We have 
 \begin{equation}
\ps{i} \ta{i} (I_\Lambda + \un \xa{i}{-1}) \subset I_\Lambda,
 \end{equation}
 if $\nu\a_i=\a_i$ and
 \begin{equation}
\ps{i} \ta{i} (I_\Lambda + \un \xa{i}{-\frac{1}{2}}) \subset I_\Lambda,
 \end{equation}
 if $\nu\a_i\neq\a_i$.
 \end{lem}
 
 \begin{proof}
 In view of Lemma \ref{lempt01}, it suffices to show that
 \be
 \ps{i} \ta{i}(J) \subset I_\Lambda.
 \ee
 As with Lemma \ref{lempt0}, we begin with the case that $\nu \alpha_i = \alpha_i$.  First, applying $\psi_{\gamma_i,\theta_i} \ta{i}$ to $R(\alpha_i,\alpha_i|t)$, we have that:
\begin{eqnarray*}
\ps{i} \ta{i}(R^0(\alpha_i,\alpha_i|t)) &=&\ps{i} \ta{i}\left(\sum_{\substack{  m_1,m_2\in\ZZ_{<0} \\ m_1+m_2=-t}}\xa{i}{m_1}\xa{i}{m_2}\right)\\
 &=& \sum_{\substack{  m_1,m_2\in\ZZ_{<0} \\ m_1+m_2=-t}}\xa{i}{m_1-1}\xa{i}{m_2-1}\xa{i}{-1}\\
 &=& \xa{i}{-1}R^0(\alpha_i,\alpha_i | t+2) + aR^0(\alpha_i,\alpha_i |2) \in I_\Lambda \\
 \end{eqnarray*}
 for some $a \in \un$. Now, we apply $\psi_{\gamma_i,\theta_i} \ta{i}$ to $R(\alpha_j,\alpha_j|t)$, where $i \neq j$. If $\langle \alpha_i, \alpha_j \rangle = 0$, we have that 
\be
\ps{i} \ta{i}(R^0(\alpha_j,\alpha_j|t)) = \xa{i}{-1}R^0(\alpha_j,\alpha_j | t) \in I_\Lambda.
\ee
 If $\langle \alpha_i, \alpha_j \rangle =-1$, we have two cases to consider: the case where $\nu \alpha_j = \alpha_j$ and the case where $\nu \alpha_j \neq \alpha_j$. 
 If $\nu \alpha_j = \alpha_j$, we have
\begin{eqnarray*}
\ps{i} \ta{i}(R^0(\alpha_j,\alpha_j|t)) &=& \ps{i} \ta{i} \left(\sum_{\substack{  m_1,m_2\in\ZZ_{<0} \\ m_1+m_2=-t}}\xa{j}{m_1}\xa{j}{m_2}\right)\\
& = &\left(\sum_{\substack{  m_1,m_2\in\ZZ_{<0} \\ m_1+m_2=-t}}\xa{j}{m_1+1}\xa{j}{m_2+1}\right)\xa{i}{-1}\\
& = &\xa{i}{-1} \left(\sum_{\substack{  m_1,m_2\in\ZZ_{<0} \\ m_1+m_2=-t}}\xa{j}{m_1+1}\xa{j}{m_2+1}\right) \\
&& + \left(\sum_{\substack{  m_1,m_2\in\ZZ_{<0} \\ m_1+m_2=-t}}\xa{j}{m_1+1}\xaa{i}{j}{m_2}\right)\epsilon_{C_0}(\alpha_j,\alpha_i)\\
&& + \left(\sum_{\substack{  m_1,m_2\in\ZZ_{<0} \\ m_1+m_2=-t}}\xaa{i}{j}{m_1}\xa{j}{m_2+1}\right)\epsilon_{C_0}(\alpha_j,\alpha_i)\\
&=& \xa{i}{-1}R^0(\alpha_j,\alpha_j|t) + [R^0(\alpha_j,\alpha_j|t-1),\xa{i}{0}] + b \in I_\Lambda
 \end{eqnarray*} 
 for some $b \in \un \on_+$. If $\nu \alpha_j \neq \alpha_j$, we have
  \begin{eqnarray}
\ps{i} \ta{i}(R^0(\alpha_j,\alpha_j|t)) &=& \ps{i} \ta{i} \left(\sum_{\substack{  m_1,m_2\in\frac{1}{2}\ZZ_{<0} \\ m_1+m_2=-t}}\xa{j}{m_1}\xa{j}{m_2}\right)\\
& = &\left(\sum_{\substack{  m_1,m_2\in\frac{1}{2}\ZZ_{<0} \\ m_1+m_2=-t}}\xa{j}{m_1+1}\xa{j}{m_2+1}\right)\xa{i}{-1}\\
& = &\xa{i}{-1} \left(\sum_{\substack{  m_1,m_2\in\frac{1}{2}\ZZ_{<0} \\ m_1+m_2=-t}}\xa{j}{m_1+1}\xa{j}{m_2+1}\right) \\
&& -  \epsilon_{C_0}(\alpha_j,\alpha_i)\left(\sum_{\substack{  m_1,m_2\in\frac{1}{2}\ZZ_{<0} \\ m_1+m_2=-t}}\xa{j}{m_1+1}\xaa{i}{j}{m_2}\right)\\
&& - \epsilon_{C_0}(\alpha_j,\alpha_i)\left(\sum_{\substack{  m_1,m_2\in\frac{1}{2}\ZZ_{<0} \\ m_1+m_2=-t}}\xaa{i}{j}{m_1}\xa{j}{m_2+1}\right)\\
& = & \xa{i}{-1} R^0(\alpha_j,\alpha_j|t-2) \\
&&  \label{annoyingcomp1} -\epsilon_{C_0}(\alpha_j,\alpha_i)\left(\sum_{\substack{  m_1,m_2\in\frac{1}{2}\ZZ_{<0} \\ m_1+m_2=-t}}\xa{j}{m_1+1}\xaa{i}{j}{m_2}\right)\\
&&\label{annoyingcomp2} - \epsilon_{C_0}(\alpha_j,\alpha_i)\left(\sum_{\substack{  m_1,m_2\in\frac{1}{2}\ZZ_{<0} \\ m_1+m_2=-t}}\xaa{i}{j}{m_1}\xa{j}{m_2+1}\right)
 \end{eqnarray} 
 where $c \in \mathbb{C}$.
We now analyze the last two sums (\ref{annoyingcomp1}) and (\ref{annoyingcomp2}). First, for (\ref{annoyingcomp1}), we have 
\begin{eqnarray*}
\left(\sum_{\substack{  m_1,m_2\in\frac{1}{2}\ZZ_{<0} \\ m_1+m_2=-t}}\xa{j}{m_1}\xaa{i}{j}{m_2}\right) &=& \left(\sum_{\substack{  m_1,m_2\in\frac{1}{2}\ZZ_{<0} \\ m_1+m_2=-t+1}}\xa{j}{m_1}\xaa{i}{j}{m_2}\right)\\
&& + \xa{j}{\frac{1}{2}}\xaa{i}{j}{-t+\frac{1}{2}} + \xa{j}{0}\xaa{i}{j}{-t+1}
\end{eqnarray*}
Notice that if $\langle \nu \alpha_j, \a_i+\a_j \rangle \ge 0$, we have
\begin{eqnarray*}
\xa{j}{\frac{1}{2}}\xaa{i}{j}{-t+\frac{1}{2}} + \xa{j}{0}\xaa{i}{j}{-t+1} &=& \xaa{i}{j}{-t+\frac{1}{2}}\xa{j}{\frac{1}{2}} + \xaa{i}{j}{-t+1}\xa{j}{0}\\
 &\in& \un \on_+
\end{eqnarray*} 
if $\langle \nu \alpha_j, \a_i+\a_j \rangle = -1$, we have
\begin{eqnarray*}
\lefteqn{\xa{j}{\frac{1}{2}}\xaa{i}{j}{-t+\frac{1}{2}} + \xa{j}{0}\xaa{i}{j}{-t+1}}\\
&=& -x_{\nu\alpha_j}^{\hnu}\left({\frac{1}{2}}\right)\xaa{i}{j}{-t+\frac{1}{2}} + x_{\nu\alpha_j}^{\hnu}(0)\xaa{i}{j}{-t+1}\\
&=& -\epsilon_{C_0}(\nu \a_j, \a_i + \a_j)x_{\nu\alpha_j + \alpha_i + \alpha_j}^{\hnu}(-t+1) + \xaa{i}{j}{-t+\frac{1}{2}}x_{\nu\alpha_j}^{\hnu}\left({\frac{1}{2}}\right)\\
&& + \epsilon_{C_0}(\nu \a_j, \a_i + \a_j)x_{\nu\alpha_j + \alpha_i + \alpha_j}^{\hnu}(-t+1) + \xaa{i}{j}{-t+1}x_{\nu\alpha_j}^{\hnu}(0)\\
& \in & \un \on_+
\end{eqnarray*} 
so that (\ref{annoyingcomp1}) can be expressed as 
\be
\epsilon_{C_0}(\alpha_j,\alpha_i)\left(\sum_{\substack{  m_1,m_2\in\frac{1}{2}\ZZ_{<0} \\ m_1+m_2=-t+1}}\xa{j}{m_1+1}\xaa{i}{j}{m_2}\right) + b
\ee
 for some $b \in \un \on_+$.
For the second sum (\ref{annoyingcomp2}), we have 
\begin{eqnarray*}
\lefteqn{\sum_{\substack{  m_1,m_2\in\frac{1}{2}\ZZ_{<0} \\ m_1+m_2=-t}}\xaa{i}{j}{m_1}\xa{j}{m_2+1}}\\
&=& \sum_{\substack{  m_1,m_2\in\frac{1}{2}\ZZ_{<0} \\ m_1+m_2=-t+1}}\xaa{i}{j}{m_1}\xa{j}{m_2} + \xaa{i}{j}{-t+1}\xa{j}{0} + \xaa{i}{j}{-t+\frac{1}{2}}\xa{j}{\frac{1}{2}}\\
&=&\sum_{\substack{  m_1,m_2\in\frac{1}{2}\ZZ_{<0} \\ m_1+m_2=-t+1}}\xaa{i}{j}{m_1}\xa{j}{m_2} + b
\end{eqnarray*}
for some $b \in \un \on_+$. Putting these results together, we see that the sums (\ref{annoyingcomp1}) and (\ref{annoyingcomp2}) can be combined and written as
\be
[R(\a_j, \a_j|t-1),\xa{i}{0}] + b
\ee
for some $b \in \un \on_+$, and so we have
\be
\ps{i} \ta{i}(R^0(\alpha_j,\alpha_j|t)) \in I_\Lambda
\ee
 as desired.

 We now examine the case when $\nu \alpha_i \neq \alpha_i$.
 Applying $\psi_{\gamma_i,\theta_i} \ta{i}$ to $R(\alpha_i,\alpha_i|t)$, we have that:
\begin{eqnarray*}
\ps{i} \ta{i}(R^0(\alpha_i,\alpha_i|t)) &=&\ps{i} \ta{i}\left(\sum_{\substack{  m_1,m_2\in\frac{1}{2}\ZZ_{<0} \\ m_1+m_2=-t}}\xa{i}{m_1}\xa{i}{m_2}\right)\\
 &=& \sum_{\substack{  m_1,m_2\in\frac{1}{2}\ZZ_{<0} \\ m_1+m_2=-t}}\xa{i}{m_1-\frac{1}{2}}\xa{i}{m_2-\frac{1}{2}}\xa{i}{-\frac{1}{2}}\\
 &=& \xa{i}{-\frac{1}{2}}R^0(\alpha_i,\alpha_i | t+1) + aR^0(\alpha_i,\alpha_i |1) \in I_\Lambda \\
 \end{eqnarray*}
 for some $a \in \un$. Now, we apply $\psi_{\gamma_i,\theta_i} \ta{i}$ to $R(\alpha_j,\alpha_j|t)$, where $i \neq j$. If $\langle \alpha_i, \alpha_j \rangle = 0$, we have that 
\begin{equation*}\begin{aligned}
\ps{i} \ta{i}(R^0(\alpha_j,\alpha_j|t)) &= \xa{i}{-\frac{1}{2}}R^0(\alpha_j,\alpha_j | t) \in I_\Lambda
 \end{aligned}\end{equation*} 
 if $\langle \a_i, \nu \a_j \rangle \ge 0$ and 
\begin{equation*}\begin{aligned}
\ps{i} \ta{i}(R^0(\alpha_j,\alpha_j|t)) &= \ps{i} \ta{i} \left(\sum_{\substack{  m_1,m_2\in\frac{1}{2}\ZZ_{<0} \\ m_1+m_2=-t}}\xa{j}{m_1}\xa{j}{m_2}\right)\\
& = \left(\sum_{\substack{  m_1,m_2\in\frac{1}{2}\ZZ_{<0} \\ m_1+m_2=-t}}\xa{j}{m_1+\frac{1}{2}}\xa{j}{m_2+\frac{1}{2}}\right)\xa{i}{-\frac{1}{2}}\\
& = -\left(\sum_{\substack{  m_1,m_2\in\frac{1}{2}\ZZ_{<0} \\ m_1+m_2=-t}}\xa{j}{m_1+\frac{1}{2}}\xa{j}{m_2+\frac{1}{2}}\right)x_{\nu \a_i}^{\hnu} \left(-\frac{1}{2}\right)\\
& = -x_{\nu \a_i}^{\hnu} \left(-\frac{1}{2}\right) \left(\sum_{\substack{  m_1,m_2\in\frac{1}{2}\ZZ_{<0} \\ m_1+m_2=-t}}\xa{j}{m_1+\frac{1}{2}}\xa{j}{m_2+\frac{1}{2}}\right) \\
& + \frac{1}{2}\epsilon_{C_0}(\alpha_j,\nu\alpha_i)\left(\sum_{\substack{  m_1,m_2\in\frac{1}{2}\ZZ_{<0} \\ m_1+m_2=-t}}\xa{j}{m_1+\frac{1}{2}}x_{\alpha_j + \nu \a_i}^{\hnu}({m_2})\right)\\
& + \frac{1}{2}\epsilon_{C_0}(\alpha_j,\nu\alpha_i)\left(\sum_{\substack{  m_1,m_2\in\frac{1}{2}\ZZ_{<0} \\ m_1+m_2=-t}}x_{\nu \a_i}^{\hnu}({m_1+\frac{1}{2}})\xa{j}{m_2}\right)\\
&= -x_{\nu \a_i}^{\hnu}\left(\frac{1}{2}\right)R^0(\alpha_j,\alpha_j|t-1) + [R^0(\alpha_j,\alpha_j|t-\frac{1}{2}),x_{\nu \a_i}^{\hnu}({0})] + b \in I_\Lambda
\end{aligned} \end{equation*}  
 for some $b \in \un \on_+$
  if $\langle \a_i, \nu \a_j \rangle = -1$ (or equivalently $\langle \nu \a_i,  \a_j \rangle = -1$).
  
As before, if $\langle \alpha_i, \alpha_j \rangle =-1$, we have two cases to consider: the case where $\nu \alpha_j = \alpha_j$ and the case where $\nu \alpha_j \neq \alpha_j$. 
  If $\nu \alpha_j = \alpha_j$, we have
 \begin{eqnarray*}
\ps{i} \ta{i}(R^0(\alpha_j,\alpha_j|t)) &=& \ps{i} \ta{i} \left(\sum_{\substack{  m_1,m_2\in\ZZ_{<0} \\ m_1+m_2=-t}}\xa{j}{m_1}\xa{j}{m_2}\right)\\
& = &\left(\sum_{\substack{  m_1,m_2\in\ZZ_{<0} \\ m_1+m_2=-t}}\xa{j}{m_1+1}\xa{j}{m_2+1}\right)\xa{i}{-\frac{1}{2}}\\
& = &\xa{i}{-\frac{1}{2}} \left(\sum_{\substack{  m_1,m_2\in\ZZ_{<0} \\ m_1+m_2=-t}}\xa{j}{m_1+1}\xa{j}{m_2+1}\right) \\
&& + \epsilon_{C_0}(\alpha_j,\alpha_i)\left(\sum_{\substack{  m_1,m_2\in\ZZ_{<0} \\ m_1+m_2=-t}}\xa{j}{m_1+1}\xaa{i}{j}{m_2+\frac{1}{2}}\right)\\
&& + \epsilon_{C_0}(\alpha_j,\alpha_i)\left(\sum_{\substack{  m_1,m_2\in\ZZ_{<0} \\ m_1+m_2=-t}}\xaa{i}{j}{m_1+\frac{1}{2}}\xa{j}{m_2+1}\right)\\
&=& \xa{i}{-1}R^0(\alpha_j,\alpha_j|t) + [R^0(\alpha_j,\alpha_j|t-1),\xa{i}{\frac{1}{2}}] + b \in I_\Lambda
 \end{eqnarray*} 
 for some $b \in \un \on_+$.
 
If $\nu \alpha_j \neq \alpha_j$, we have examine two cases: when $\langle \nu \a_i, \a_j \rangle = 0$ and when 
 $\langle \nu \a_i, \a_j \rangle = -1$. If $\langle \nu \a_i, \a_j \rangle = 0$, we have
  \begin{eqnarray*}
\ps{i} \ta{i}(R^0(\alpha_j,\alpha_j|t)) &=& \ps{i} \ta{i} \left(\sum_{\substack{  m_1,m_2\in\frac{1}{2}\ZZ_{<0} \\ m_1+m_2=-t}}\xa{j}{m_1}\xa{j}{m_2}\right)\\
& = &\left(\sum_{\substack{  m_1,m_2\in\frac{1}{2}\ZZ_{<0} \\ m_1+m_2=-t}}\xa{j}{m_1+\frac{1}{2}}\xa{j}{m_2+\frac{1}{2}}\right)\xa{i}{-\frac{1}{2}}\\
& = &\xa{i}{-\frac{1}{2}} \left(\sum_{\substack{  m_1,m_2\in\frac{1}{2}\ZZ_{<0} \\ m_1+m_2=-t}}\xa{j}{m_1+\frac{1}{2}}\xa{j}{m_2+\frac{1}{2}}\right) \\
&& + \frac{1}{2}\epsilon_{C_0}(\alpha_j,\alpha_i)\left(\sum_{\substack{  m_1,m_2\in\frac{1}{2}\ZZ_{<0} \\ m_1+m_2=-t}}\xa{j}{m_1+\frac{1}{2}}\xaa{i}{j}{m_2}\right)\\
&& +\frac{1}{2} \epsilon_{C_0}(\alpha_j,\alpha_i)\left(\sum_{\substack{  m_1,m_2\in\frac{1}{2}\ZZ_{<0} \\ m_1+m_2=-t}}\xaa{i}{j}{m_1}\xa{j}{m_2+\frac{1}{2}}\right)\\
&=& \xa{i}{-\frac{1}{2}}R^0(\alpha_j,\alpha_j|t) + [R^0(\alpha_j,\alpha_j|t-\frac{1}{2}),\xa{i}{0}] + b \in I_\Lambda
 \end{eqnarray*} 
 for some $b \in \un \on_+$. If $\langle \nu \a_i, \a_j \rangle = -1$, we have
   \begin{eqnarray*}
\ps{i} \ta{i}(R^0(\alpha_j,\alpha_j|t)) &=& \ps{i} \ta{i} \left(\sum_{\substack{  m_1,m_2\in\frac{1}{2}\ZZ_{<0} \\ m_1+m_2=-t}}\xa{j}{m_1}\xa{j}{m_2}\right)\\
& = &\left(\sum_{\substack{  m_1,m_2\in\frac{1}{2}\ZZ_{<0} \\ m_1+m_2=-t}}\xa{j}{m_1+1}\xa{j}{m_2+1}\right)\xa{i}{-\frac{1}{2}}\\
& = &\xa{i}{-\frac{1}{2}} \left(\sum_{\substack{  m_1,m_2\in\frac{1}{2}\ZZ_{<0} \\ m_1+m_2=-t}}\xa{j}{m_1+1}\xa{j}{m_2+1}\right) \\
&& + \epsilon_{C_0}(\alpha_j,\alpha_i)\left(\sum_{\substack{  m_1,m_2\in\frac{1}{2}\ZZ_{<0} \\ m_1+m_2=-t}}\xa{j}{m_1+1}\xaa{i}{j}{m_2+\frac{1}{2}}\right)\\
&& + \epsilon_{C_0}(\alpha_j,\alpha_i)\left(\sum_{\substack{  m_1,m_2\in\frac{1}{2}\ZZ_{<0} \\ m_1+m_2=-t}}\xaa{i}{j}{m_1+\frac{1}{2}}\xa{j}{m_2+1}\right)\\
&=& \xa{i}{-\frac{1}{2}}R^0(\alpha_j,\alpha_j|t) + [R^0(\alpha_j,\alpha_j|t-\frac{3}{2}),\xa{i}{0}] + b \in I_\Lambda
 \end{eqnarray*} 
  for some $b \in \un \on_+$, completing our proof.

 \end{proof}

Consider the maps
 \be\begin{aligned}
 e_{\a_i}&:V_L^T\to V_L^T
  \end{aligned}\ee
  and their restriction to the principal subspace $W_L^T\subset V_L^T$ where 
  \be\begin{aligned}
  e_{\a_i}\cdot 1&=\frac{2}{\sigma(\a_i)}\xa{i}{-1}\cdot 1 \text{ if } \nu \a_i = \a_i\\
    e_{\a_i}\cdot 1&=\frac{2}{\sigma(\a_i)}\xa{i}{-\frac{1}{2}}\cdot 1\text{ if } \nu \a_i \neq\a_i,
 \end{aligned}\ee
 and
 \be
 e_{\a_i}\xbet{n}=C(\a_i,\b)\xbet{n-\left<\beta_{(0)},\a_i\right>}e_{\a_i}.
 \ee

 Recall that we have a basis dual to the $\ZZ$-basis with respect to $\left<\cdot,\cdot\right>$, of $L$ given by $\{\lambda_i |1\leq i\leq l\}$ as described in (\ref{dualbasis}). For $\lambda\in\{\lambda_i |1\leq i\leq l\}$, we now introduce the map 
 \be
 \Delta^{T}(\lambda,-x)=(-1)^{\nu \lambda}x^{\lambda_{(0)}}E^{+}(-\lambda,x).
 \ee
 This is a twisted version of the $\Delta$-map in \cite{Li2}. Recently such twisted maps have been exploited in \cite{CalLM4}, \cite{CalMPe}, and \cite{PS}. We will make use the constant term of $\Delta^T(\lambda,-x)$ denoted by $\Delta_c^T(\lambda,-x)$. Let us explore the action of $\Delta_c^T(\lambda,-x)$ on the operators $\xbet{m}$ where $\b\in\Delta_+$. This calculation will be made in two parts, the first being when $\nu\a_i=\a_i$ and thus $\nu\lambda_i=\lambda_i$:
 \begin{equation}\begin{aligned}
 \Delta_c^T(\lambda_i,-x)(\xbet{m}\cdot 1)&=\text{Coeff}_{x^0x_1^{-m-1}}\left(\Delta^T(\lambda_i,-x)Y^{\hnu}(e^{\b},x_1)\right)\\
 &=\theta_i(\beta)\text{Coeff}_{x_1^{\left<\lambda_i,\b\right>-m-1}}Y^{\hnu}(e^{\b},x_1)\\
 &=\theta_i(\beta)\xbet{m+\left<\lambda_i,\b\right>}\cdot 1\\
 &=\ta{i}(\xbet{m})\cdot 1,
 \end{aligned}\end{equation}
 where we have used the fact that
 \be
 E^{+}(-\lambda_i,x)E^{-}(-\b,x_1)=\left(1-\frac{x_1}{x}\right)^{\left<\lambda_i,\b\right>}E^{-}(-\b,x_1)E^{+}(-\lambda_i,x).
 \ee 
 Now for the case when $\nu\lambda_i\neq\lambda_i$:
  \begin{equation}\begin{aligned}
 \Delta_c^T(\lambda_i,-x)(\xbet{m}\cdot 1)&=\text{Coeff}_{x^0x_1^{-m-1}}\left(\Delta^T(\lambda_i,-x)Y^{\hnu}(e^{\b},x_1)\right)\\
 &=\theta_i(\beta)\text{Coeff}_{x_1^{\frac{1}{2}\left(\left<\lambda_i,\b\right>+\left<\nu\lambda_i,\b\right>\right)-m-1}}Y^{\hnu}(e^{\b},x_1)\\
 &=\theta_i(\beta)\xbet{m+\frac{1}{2}\left(\left<\lambda_i,\b\right>+\left<\nu\lambda_i,\b\right>)\right)}\cdot 1\\
 &=\ta{i}(\xbet{m})\cdot 1,
 \end{aligned}\end{equation}
 where we have used
 \be
 E^{+}(-\lambda_i,x)E^{-}(-\b,x_1)=\left(1-\frac{x_1^{\frac{1}{2}}}{x^{\frac{1}{2}}}\right)^{\left<\lambda_i,\b\right>}\left(1+\frac{x_1^{\frac{1}{2}}}{x^{\frac{1}{2}}}\right)^{\left<\nu\lambda_i,\b\right>}E^{-}(-\b,x_1)E^{+}(-\lambda_i,x).
 \ee 
 It follows that we have linear maps
 \begin{equation}\label{DeltaProperty}\begin{aligned}
 \Delta_c^T(\lambda_i,-x):W_L^T&\to W_L^T\\
 a\cdot 1&\mapsto \ta{i}(a)\cdot 1,
 \end{aligned}\ee
 where $a\in\un$.

 \section{The Main Results}
 
 \begin{thm}\label{presentation}We have
\begin{equation}\text{Ker}~ f_{\Lambda}=I_{\Lambda},
\end{equation}
or equivalently 
\begin{equation}\text{Ker}~ \pi_{\Lambda}=I_{\Lambda}\cdot v_{\Lambda}^{N}.\end{equation}
\end{thm}

\begin{proof} The inclusion
\be
I_\Lambda \cdot v_\Lambda^N \subset \mathrm{Ker} \ \pi_\Lambda
\ee
is trivial. The remainder of the proof will be for the reverse inclusion.
Suppose that $a \in \ker \pi_\Lambda \setminus I_\Lambda \cdot v_\Lambda^N$. We may assume that $a$ is homogeneous with respect to all gradings and is a nonzero such element with smallest 
possible total charge. (Consequently, $a$ has a nonzero charge for some $(\lambda^{(i)})_{(0)}$.) Among all elements of smallest possible total charge, we may assume that $a$ is also of smallest $L^{\hnu} (0)$-weight.
We first show that either 
\begin{equation}
a \in I_\Lambda + \un \xa{i}{-1}
\end{equation}
where $\nu\a_i=\a_i$, or 
\begin{equation}
a \in I_\Lambda + \un \xa{i}{-\frac{1}{2}},
\end{equation}
where $\nu\a_i\neq \a_i$.
Suppose not. Then
 by Lemma \ref{Containment} we have that 
\begin{equation}
\ta{i}(a) \cdot v_{\Lambda}^N \notin I_\Lambda \cdot v_\Lambda^N.
\end{equation}
We also have that 
\begin{equation}
a \cdot 1 = \pi_\Lambda (a \cdot v_\Lambda^N) = 0,
\end{equation}
and so using (\ref{DeltaProperty}), we have
\begin{equation}
\pi_\Lambda (\ta{i}(a) \cdot v_{\Lambda}^N) = \ta{i}(a) \cdot 1 = 0.
\end{equation}
But, since 
\begin{equation}
wt(\ta{i}(a)) < wt(a)
\end{equation}
we have $\ta{i}(a) \in I_\Lambda$, contradicting the minimality of $a$, and so
\begin{equation}
a \in I_\Lambda + \un \xa{i}{-1}
\end{equation}
in the case that $\nu\a_i=\a_i$ or 
\begin{equation}
a \in I_\Lambda + \un \xa{i}{-\frac{1}{2}},
\end{equation}
if $\nu\a_i\neq \a_i$.
First, we suppose that $\a_i$ is fixed under $\nu$ and $a \in I_\Lambda + \un \xa{i}{-1}$, so that
$$
a = b + c\xa{i}{-1}
$$
for some $b \in I_\Lambda$ and $c \in \un$. Notice that $b$ and $c\xa{i}{-1}$ are of the same total $L^{\hnu}(0)$ weight and total charge as $a$.  We have that 
\begin{eqnarray*}
c \xa{i}{-1} \cdot 1 = (a-b) \cdot 1 =0,
\end{eqnarray*}
so that $c\xa{i}{-1} \cdot v_\Lambda^n \in \ker \pi_\Lambda \setminus I_\Lambda \cdot v_\Lambda^N$. We also have that 
\begin{equation} 
c\xa{i}{-1}\cdot 1 = e_{\alpha_i}(\tau_{\alpha_i,\theta_i}(c) \cdot 1) = 0
\end{equation}
so that, by the injectivity of $e_{\alpha_i}$, we have that 
\begin{equation}
\tau_{\alpha_i,\theta_i}(c) \cdot 1 = 0.
\end{equation}
Since  $\tau_{\alpha_i,\theta_i}(c)$ has lower total charge than $a$, we have that 
\begin{equation}
\tau_{\alpha_i,\theta_i}(c) \in I_\Lambda.
\end{equation}
In view of Lemma \ref{Containment}, we have that 
\begin{equation}
\tau_{-\gamma_i,\theta_i^{-1}}(\tau_{\alpha_i,\theta_i}(c)) \in I_\Lambda + \un \xa{i}{-1}.
\end{equation}
Finally, applying Lemma \ref{CompContainment} immediately gives
\begin{equation}
\psi_{\gamma_i,\tau_i}\tau_{\gamma_i,\theta_i}\tau_{-\gamma_i,\theta_i^{-1}}(\tau_{\alpha_i,\theta_i}(c)) = c\xa{i}{-1} \in I_\Lambda.
\end{equation}
Thus, we have that 
\begin{equation}
a = b + c\xa{i}{-1} \in I_\Lambda,
\end{equation}
which is a contradiction. A similar proof shows that if $\a_i$ is not fixed under $\nu$
\begin{equation}
a = b + c\xa{i}{-\frac{1}{2}}
\end{equation}
for some $b \in I_\Lambda$ and $c \in \un$, then $a \in I_\Lambda$, thus giving a contradiction.
\end{proof}

 \begin{thm}\label{exact}For $i$ chosen from a subset of $\{1,2,3,\cdots, l\}$ so that $(\a_i)_{(0)}$ creates a complete and non-repeating list, we have the following short exact sequences
\begin{equation}
0\to W_L^T \xrightarrow{e_{\alpha_i}}W_L^T \xrightarrow{\Delta_c^T(\lambda_i,-x)}W_L^T\to 0
\ee

\end{thm}

\begin{proof}
It is clear that for all chosen $i$ $e_{\a_i}$, is injective and $\Delta_c^T(\lambda_i,-x)$ is surjective. Now we take $w=a\cdot 1\in\text{ker }\Delta_c^T(\lambda_i,-x)$ for $a\in\un$. So we have
\be
0=\Delta_c^T(\lambda_i, -x)w=\ta{1}(a)\cdot 1,
\ee 
and thus $\ta{i}(a)\in I_{\Lambda}$ by Theorem \ref{presentation}. This implies that 
\be
a\in I_{\Lambda}+\un\xa{i}{-1},
\ee
if $\nu\a_i=\a_i$ and 
\be
a\in I_{\Lambda}+\un\xa{i}{-\frac{1}{2}},
\ee
if $\nu\a_i\neq \a_i$, by Lemma \ref{Containment}. So we have the condition that $w=a\cdot 1\in\text{ker }\Delta_c^T(\lambda_i,-x)$ is equivalent to the condition that $a\in I_{\Lambda}+\un\xa{i}{-1}$ ( or $I_{\Lambda}+\un\xa{i}{-\frac{1}{2}}$ as appropriate).

Now suppose $w=a\cdot 1\in\text{Im }e_{\a_i}$ for $a\in\un$. So we can write $w=b\xa{i}{-1}\cdot 1$ for $b\in\un$ if $\nu\a_i=\a_i$ or $w=b^{'}\xa{i}{-\frac{1}{2}}\cdot 1$ for $b^{'}\in\un$ if $\nu\a_i\neq\a_i$ and thus $a\cdot 1=b\xa{i}{-1}\cdot 1\in I_{\Lambda}+\un\xa{i}{-1}$ $a\cdot 1=b^{'}\xa{i}{-\frac{1}{2}}\cdot 1\in I_{\Lambda}+\un\xa{1}{-\frac{1}{2}}$ as appropriate. This implies that the condition that $w=a\cdot 1\in\text{Im}e_{\a_i}$ is equivalent to the condition that $a\in I_{\Lambda}+\un\xa{i}{-1}$ (or $I_{\Lambda}+\un\xa{i}{-\frac{1}{2}}$ as appropriate). So we have
\be
w=a\cdot 1\in\text{ker }\Delta_c^T(\lambda_i,-x) \text{ if and only if } w=a\cdot 1\in\text{Im }e_{\a_i},
\ee
and thus 
\be
\text{ker }\Delta_c^T(\lambda_i,-x)=\text{Im }e_{\a_i}.
\ee

\end{proof}

We use the multi-grading described earlier, and define the multigraded dimension of $W_L^T$ by:
\begin{equation}\label{character1}
\chi({\bf x};q) = \mathrm{tr}|_{W_L^T}x_1^{(\lambda^{(i_1)})_{(0)}}\cdots x_k^{(\lambda^{(i_k)})_{(0)}} q^{2\hat{L}^{\hnu}(0)} \in q^{\frac{\text{dim}(\h_{(1)})}{16}}\mathbb{C}[[{\bf x},q]],
\end{equation}
where the $\lambda^{(i_j)}$ are described in (\ref{charge1}-\ref{charge2}). We use the modification
\be\label{character2}
\chi'({\bf x};q) = q^{-\frac{\text{dim}\h_{(1)}}{16}}\chi({\bf x};q) \in \mathbb{C}[[{\bf x},q]].
\ee

\begin{cor}\label{exact2}Let $\{\a_{i_1}, \dots, \a_{i_k}\}$ be a complete set of representatives from the equivalence classes of the simple roots $\{\a_1,\dots,\a_l\}$ formed by orbits under $\nu$, $\mathbf{m}=(m_1,m_2,\dots,m_k)$. We have the following short exact sequences
\begin{equation}
0\to \left(W_L^T\right)_{(\mathbf{m}-\EE_i,n-\sum_{j=1}^k2\left<(\a_{i_j})_{(0)},(\a_i)_{(0)}\right>m_j+2)} \xrightarrow{e_{\alpha_i}}\left(W_L^T\right)_{(\mathbf{m},n)} \xrightarrow{\Delta_c^T(\lambda_i,-x)}\left(W_L^T\right)_{(\mathbf{m},n-2m_i)}\to 0,
\end{equation}
if $\nu\a_i=\a_i$ and
\begin{equation}
0\to \left(W_L^T\right)_{(\mathbf{m}-\EE_i,n-\sum_{j=1}^k2\left<(\a_{i_j})_{(0)},(\a_i)_{(0)}\right>m_j+1)} \xrightarrow{e_{\alpha_i}}\left(W_L^T\right)_{(\mathbf{m},n)} \xrightarrow{\Delta_c^T(\lambda_i,-x)}\left(W_L^T\right)_{(\mathbf{m},n-m_i)}\to 0,
\end{equation}
if $\nu\a_i\neq \a_i$. Moreover, we have
\begin{equation}
\chi'(\mathbf{x};q) = \chi'(x_1,\dots,x_{i-1},q^2x_i,x_{i+1},\dots,x_k;q) + x_iq^2 \chi'(q^{2\left<(\a_{i_1})_{(0)}(\a_{i})_{(0)}\right>}x_1,\dots,q^{2\left<(\a_{i_k})_{(0)}(\a_{i})_{(0)}\right>}x_k;q)
\end{equation}
if $\nu\a_i=\a_i$ and
\begin{equation}
\chi'(\mathbf{x};q) = \chi'(x_1,\dots,x_{i-1},qx_i,x_{i+1},\dots,x_k;q) + x_iq \chi'(q^{2\left<(\a_{i_1})_{(0)}(\a_{i})_{(0)}\right>}x_1,\dots,q^{2\left<(\a_{i_k})_{(0)}(\a_{i})_{(0)}\right>}x_k;q)
\end{equation}
if $\nu\a_i\neq\a_i$.
\end{cor}

Solving this recursion (cf. \cite{A}) and using the notation $(a;q)_n = (1-a)(1-aq)(1-aq^2)\dots (1-aq^{n-1})$, we have:

\begin{cor}
We have
\begin{equation}
\chi^{'}(\mathbf{x};q) = \sum_{{\bf m} \in (\mathbb{Z}_{\ge 0}^k)}\frac{q^\frac{{\bf m}^t A {\bf m}}{2}}{(q^{a_1};q^{a_1})_{m_1} \cdots (q^{a_k};q^{a_k})_{m_k} }x_1^{m_1}\cdots x_k^{m_k}
\end{equation}
where $a_j=2$ if $\nu\a_{i_j}=\a_{i_j}$, $a_j=1$ if $\nu\a_{i_j}\neq\a_{i_j}$ and $A=2\left(\left<(\a_{i_j})_{(0)},(\a_{i_k})_{(0)}\right>\right)$.
\end{cor}
Explicitly,we have:\\
For $\frak{g} = A_{2n-1}$, we have the $n \times n$ matrix
\begin{center}
$
A = \begin{bmatrix}
    2      & -1     & 0      & 0      & 0      & \ldots & \ldots & 0\\
    -1     &  2     & -1     & 0      & 0      & \ldots & \ldots & 0\\
    0      & -1     &  2     & -1     & 0      & \ldots & \ldots & 0\\
    \vdots & \vdots & \vdots & \ddots & \vdots & \ldots & \ldots & \vdots \\
    \vdots & \vdots & \vdots & \vdots & \ddots & \ldots & \ldots & \vdots \\
    \vdots & \vdots & \vdots & \vdots & \vdots & 2      & -1     & 0 \\
    \vdots & \vdots & \vdots & \vdots & \vdots & -1     & 2      & -2 \\
    \vdots & \vdots & \vdots & \vdots & \vdots & 0      & -2     & 4 \\
\end{bmatrix}$\\
\end{center}
Note: This is a symmetrized Cartan matrix for the Lie algebra $C_{n}$ with the last row multiplied by $2$.
For $\frak{g} = D_n$, we have the $(n-1) \times (n-1)$ matrix\\
\begin{center}
$
A = \begin{bmatrix}
    4      & -2     & 0      & 0      & 0      & \ldots & \ldots & 0\\
    -2     &  4     & -2     & 0      & 0      & \ldots & \ldots & 0\\
    0      & -2     &  4     & -2     & 0      & \ldots & \ldots & 0\\
    \vdots & \vdots & \vdots & \ddots & \vdots & \ldots & \ldots & \vdots \\
    \vdots & \vdots & \vdots & \vdots & \ddots & \ldots & \ldots & \vdots \\
    \vdots & \vdots & \vdots & \vdots & \vdots & 4      & -2     & 0 \\
    \vdots & \vdots & \vdots & \vdots & \vdots & -2     & 4      & -2 \\
    \vdots & \vdots & \vdots & \vdots & \vdots & 0      & -2     & 2 \\
\end{bmatrix}$\\
\end{center}
 Note: This is a symmetrized Cartan matrix for the Lie algebra $B_{n-1}$ with the first $n-2$ rows multiplied by $2$.\\
For $\frak{g} = E_6$, we have 
 \begin{center}
$
A = \begin{bmatrix}
    2 & -1 & 0 & 0\\
    -1 & 2 & -2 & 0\\
    0 & -2 & 4 & -2\\
    0 & 0 & -2 & 4\\
\end{bmatrix}$\\
\end{center}
 Note: This is a symmetrized Cartan matrix for the Lie algebra $F_4$ with the last $2$ rows multiplied by $2$.\\

Although there is a unique level one standard module for the algebras $A_{2n-1}^{(2)}$, $D_n^{(2)}$, and $E_6^{(2)}$ as well as associated twisted $V_L$ modules, we will also consider certain shifted principal subspaces and their characters, in parallel with \cite{CalLM3}, \cite{CalMPe}, and \cite{PS}. We will only include the general outline of the construction of the shifted subspaces along with their characters. For a detailed treatment see \cite{PS}. 

For $\gamma\in\mathbb{Q}\otimes_{\ZZ}L$, we form an isomorphic shifted copy of $\hat{\g}[\hnu]$ denoted by 
\be
\hat{\g}[\hnu_{\gamma}].\ee
Further, we have a representation of $\hat{\g}[\hnu_{\gamma}]$ on a shifted version of $V_L^T$ denoted by $V_L^{T,\gamma}$, with vertex operators given by 
\be
Y^{\hnu,\gamma}(\iota(e_{\alpha}),x)=Y^{\hnu}(\iota(e_{\alpha}),x)x^{\left<\alpha_{(0)},\gamma\right>},\ee
whose highest weight vector is we denote by $v_{\Lambda}^{\gamma}$. Finally, we define the $\gamma$-shifted principal subspace
\be 
W_L^{T,\gamma}=U(\overline{\mathfrak{n}}[\hnu_{\gamma}])\cdot v_{\Lambda}^{\gamma}\ee
where $\overline{\mathfrak{n}}[\hnu_{\gamma}]$ is a shifted version of $\overline{\mathfrak{n}}[\hnu]$.
 
 Now we move to our particular set up. Recall $\gamma_i=(\lambda_i)_{(0)}$ and let $\gamma_{i_1},\dots,\gamma_{i_k}$ be a complete list of such elements. For each $j\in\{i_1,\dots,i_k\}$ we consider the subspace $W_L^{\gamma_i}$ with its character $\chi_{j}^{'}(\mathbf{x};q)$ defined in parallel with (\ref{character1}) and (\ref{character2}). These shifted characters are related to $\chi^{'}(\mathbf{x};q)$ by
 \be
\chi_{j}^{'}(\mathbf{x};q)=\chi^{'}(x_1,\dots,q^a x_j,\dots, x_k;q)
\ee
where $a=2$ if $\nu\a_{i_j}=\a_{i_j}$ and $a=1$ otherwise. This implies that 
\be
\chi_j^{'}(\mathbf{x};q) = \sum_{{\bf m} \in (\mathbb{Z}_{\ge 0}^k)}\frac{q^{\frac{{\bf m}^t A {\bf m}}{2}+am_j}}{(q^{a_1};q^{a_1})_{m_1} \cdots (q^{a_k};q^{a_k})_{m_k} }x_1^{m_1}\cdots x_k^{m_k}
\end{equation}
where $a_r=2$ if $\nu\a_{i_r}=\a_{i_r}$, $a_r=1$ if $\nu\a_{i_j}\neq\a_{i_j}$, $a=2$ if $\nu\a{i_j}=\a_{i_j}$, $a=1$ if $\nu\a_{i_j}\neq\a_{i_j}$, and $A=2\left(\left<(\a_{i_j})_{(0)},(\a_{i_k})_{(0)}\right>\right)$.

\vspace{.3in}

\vspace{.2in}

\noindent{\small \sc Department of Mathematics and Computer Science, Colorado College, 
Colorado Springs, CO 80903} \\ {\em E--mail address}: michael.penn@coloradocollege.edu

\vspace{.2in}
\noindent{\small \sc Department of Mathematics and Computer Science, Ursinus College, 
Collegeville, PA 19426} \\ {\em E--mail address}: csadowski@ursinus.edu


\begin{thebibliography}{CalLM4}

\bibitem[A]{A} G. Andrews, {\em The Theory of Partitions},
  Encyclopedia of Mathematics and Its Applications, Vol. 2,
  Addison-Wesley, 1976.
  
\bibitem[AKS]{AKS} E. Ardonne, R. Kedem, and M. Stone, Fermionic 
characters and arbitrary highest-weight integrable
$\widehat{sl}_{r+1}$-modules, {\em Comm. Math. Phys.} {\bf 264} (2006) 427-464.

\bibitem[Ba]{Ba} 
I. Baranovi\'{c}, 
{\em Combinatorial bases of Feigin-Stoyanovsky's type subspaces of level 2 standard modules for $D_4^{(1)}$}, 
Comm. Algebra {\bf 39} (2011), 1007--1051.

\bibitem[BHL]{BHL} K. Barron, Y.-Z. Huang and J. Lepowsky, An
  equivalence of two constructions of permutation-twisted modules for
  lattice vertex operator algebras, {\em J. Pure Appl. Algebra}, {\bf
    210} (2007), 797--826.

\bibitem[B]{B} R. E. Borcherds, Vertex algebras, Kac-Moody algebras,
  and the Monster, {\em Proc. Natl. Acad. Sci. USA} {\bf 83} (1986),
  3068--3071.
  
\bibitem[Bu1]{Bu1} M. Butorac, Combinatorial bases of principal subspaces for affine Lie algebra of type 
  $B_2^{(1)}$, {\em J. Pure Appl. Algebra} {\bf 218} (2014) 424-447.

\bibitem[Bu2]{Bu2} M. Butorac, Quasi-particle bases of principal subspaces for the affine Lie algebras of types $B_{l}^{(1)}$ and $C_{l}^{(1)}$,
{\em Glas. Mat.}, to appear; arXiv:math.QA/1505.00450.
  
  \bibitem[BCFK]{BCFK} K. Bringmann, C. Calinescu, A. Folsom and S. Kimport,
Graded dimensions of principal subspaces and modular Andrews-Gordon-type series,
{\em Comm. in Contemp. Math.} {\bf 16} (2014).

\bibitem[C1]{C1} C. Calinescu, Principal subspaces of higher-level
  standard $\widehat{\mathfrak{sl}(3)}$-modules, {\em J. Pure
    Appl. Algebra} {\bf 210} (2007), 559–-575.

\bibitem[CalLM1]{CalLM1} C. Calinescu, J. Lepowsky and A. Milas,
  Vertex-algebraic structure of the principal subspaces of certain
  $A_{1}^{(1)}$-modules, I: level one case, {\em Internat. J. Math.}
  {\bf 19} (2008), 71--92.

\bibitem[CalLM2]{CalLM2} C. Calinescu, J. Lepowsky and A. Milas,
  Vertex-algebraic structure of the principal subspaces of certain
  $A_1^{(1)}$-modules, II: higher level case, {\em J. Pure
    Appl. Algebra} {\bf 212} (2008), 1928--1950.

\bibitem[CalLM3]{CalLM3} C. Calinescu, J. Lepowsky and A. Milas,
  Vertex-algebraic structure of the principal subspaces of level one
  modules for the untwisted affine Lie algebras of types A, D, E, {\em
    J. Algebra} {\bf 323} (2010), 167--192.

\bibitem[CalLM4]{CalLM4} C. Calinescu, J. Lepowsky and A. Milas, 
Vertex-algebraic structure of principal subspaces of standard 
$A_2^{(2)}$-modules, I, {\em Internat. J. Math.} {\bf 25} (2014).

\bibitem[CalMPe]{CalMPe} C. Calinescu, A. Milas, and M. Penn,
Vertex-Algebraic Structure of Principal Subspaces of Basic $A_{2n}^{(2)}$-Modules, {\em J. Pure Appl. Algebra} {\bf 220} (2016) 1752-1784.

\bibitem[CLM1]{CLM1} S. Capparelli, J. Lepowsky and A. Milas, The
Rogers-Ramanujan recursion and intertwining operators, {\em Comm. in
Contemp. Math.} {\bf 5} (2003), 947--966.

\bibitem[CLM2]{CLM2} S. Capparelli, J. Lepowsky and A. Milas, The
Rogers-Selberg recursions, the Gordon-Andrews identities and
intertwining operators, {\em The Ramanujan Journal} {\bf 12} (2006),
379--397.

\bibitem[DL1]{DL1} C. Dong and J. Lepowsky, The algebraic structure of
  relative twisted vertex operators, {\em J. Pure Appl. Algebra} {\bf
    110} (1996), 259--295.

\bibitem[DL2]{DL2} C. Dong and J. Lepowsky, {\em Generalized Vertex
    Algebras and Relative Vertex Operators}, Progress in Mathematics,
  Vol. 112, 1993.

\bibitem[DLM]{DLM} C. Dong, H. Li and G. Mason, Simple currents and
  extensions of vertex operator algebras, {\em Comm. in Math. Phys.} {\bf
    180} 1996, 671--707.
    
 \bibitem[DLeM]{DLeM} B. Doyon, J. Lepowsky and A. Milas, Twisted
vertex operators and Bernoulli polynomials, {\em Comm. in
Contemp. Math.}  {\bf 8} (2006), 247--307.

\bibitem[FFJMM]{FFJMM} B. Feigin, E. Feigin, M. Jimbo, T. Miwa, and
  E. Mukhin, Principal $\widehat{sl_3}$ subspaces and quantum Toda
  Hamiltonian, {\em Advanced Studies in Pure Math.} {\bf 54}
  (2009), 109-166.
 
\bibitem[FS1]{FS1} B. Feigin and A. Stoyanovsky, Quasi-particles
  models for the representations of Lie algebras and geometry of flag
  manifold; arXiv:hep-th/9308079.

\bibitem[FS2]{FS2} B. Feigin and A. Stoyanovsky, Functional models for
  representations of current algebras and semi-infinite Schubert cells
  (Russian), {\em Funktsional Anal. i Prilozhen.} {\bf 28} (1994),
  68--90; translation in: {\em Funct. Anal. Appl.} {\textbf 28}
  (1994), 55--72.

\bibitem[FHL]{FHL} I. Frenkel, Y.-Z. Huang and J. Lepowsky, On axiomatic
  approaches to vertex operator algebras and modules, {\it Memoirs
    American Math. Soc.} {\bf 104}, 1993.

\bibitem[FK]{FK} I. Frenkel and V. Kac, Basic representations of affine
  Lie algebras and dual resonance models, {\em Invent. Math.} {\bf 62}
  (1980), 23--66.

\bibitem[FLM1]{FLM1} I. Frenkel, J. Lepowsky and A. Meurman, A natural
  representation of the Fischer-Griess Monster with the modular
  function $J$ as character, {\it Proc. Natl. Acad. Sci. USA} {\bf 81}
  (1984), 3256--3260.

\bibitem[FLM2]{FLM2} I. Frenkel, J. Lepowsky and A. Meurman, Vertex
  operator calculus, in: {\em Mathematical Aspects of String Theory,
    Proc. 1986 Conference, San Diego,} ed. by S.-T. Yau, World
  Scientific, Singapore, 1987, 150--188.

\bibitem[FLM3]{FLM3} I. Frenkel, J. Lepowsky and A. Meurman, {\em
    Vertex Operator Algebras and the Monster}, Pure and Applied Math.,
  Vol. 134, Academic Press, 1988.


\bibitem [G]{G} G. Georgiev, Combinatorial constructions of modules
  for infinite-dimensional Lie algebras, I. Principal subspace, {\em
    J. Pure Appl. Algebra} { \bf 112} (1996), 247--286.



\bibitem[J1]{J1} M. Jerkovi\'{c}, 
{\em Recurrence relations for characters of affine Lie algebra $A_l^{(1)}$},  
J. Pure Appl. Algebra \textbf{213} (2009), 913--926. 

\bibitem[J2]{J2} M. Jerkovi\'{c}, 
{\em Character formulas for Feigin-Stoyanovsky's type subspaces of standard $\widetilde{\mathfrak{sl}}(3,C)$-modules}, 
Ramanujan J. \textbf{27} (2012), 357--376. 

\bibitem[JP]{JP}
M. Jerkovi\'{c}, M. Primc,
{\em Quasi-particle fermionic formulas for $(k,3)$-admissible configurations},
Cent. Eur. J. Math. \textbf{10} (2012), 703--721. 



\bibitem[K]{K} V. Kac, {\em Infinite Dimensional Lie Algebras}, $3$rd
  edition, Cambridge University Press, 1990.
  
\bibitem[Ka1]{Ka1} K. Kawasetsu, {\em The intermediate vertex subalgebras of the lattice vertex operator algebras}, 
arXiv:math.QA/1305.6463.

\bibitem[Ka2]{Ka2} K. Kawasetsu, {\em The Free Generalized Vertex Algebras and Generalized Principal Subspaces}, 
arXiv:math.QA/1502.06985.  

\bibitem[Ko]{Ko}
S. Ko\v{z}i\'{c}, 
{\em Principal subspaces for quantum affine algebra $U_{q}(A_{n}^{(1)})$}, 
J. Pure Appl. Algebra \textbf{218} (2014), 2119--2148.

\bibitem[L1]{L1} J. Lepowsky, Calculus of twisted vertex operators,
  {\em Proc.  Nat. Acad. Sci. USA} {\bf 82} (1985), 8295--8299.

\bibitem[L2]{L2} J. Lepowsky, Perspectives on vertex operators and the
  Monster, in: Proc. 1987 Symposium on the Mathematical Heritage of
  Hermann Weyl, Duke Univ., {\em Proc. Symp. Pure Math., American
    Math. Soc.} {\bf 48} (1988), 181--197.

\bibitem[LL]{LL} J. Lepowsky and H. Li, {\it Introduction to Vertex
    Operator Algebras and Their Representations}, Progress in
  Mathematics, Vol. 227, Birkh\"auser, Boston, 2003.
  
\bibitem[LP]{LP}  J. Lepowsky and M. Primc, Structure of the standard modules for the affinene
Lie algebra $A_1^{(1)}$, {\em Contemp. Math.} {\bf 46}, American Mathematical Society, Providence, 1985.
  
  \bibitem[LW]{LW} J. Lepowsky and R. L. Wilson, Construction of the
affine Lie algebra $A_1^{(1)}$, {\em Comm. Math. Phys.}  {\bf 62},
(1978) 43-53.

\bibitem[Li1]{Li1} H.-S. Li, Local systems of twisted vertex
  operators, vertex superalgebras and twisted modules, {\em
    Contemp. Math.} {\bf 193} (1996), 203--236.

\bibitem[Li2]{Li2} H.-S. Li, The physics superselection principle in
vertex operator algebra theory, {\em J. Algebra} {\bf 196} (1997),
436--457.

\bibitem[MPe]{MPe} A. Milas and M. Penn, Lattice vertex algebras and
  combinatorial bases: general case and $\mathcal{W}$-algebras, {\em
    New York J. Math.} {\bf 18} (2012), 621--650.

\bibitem[P]{P} M. Penn, Lattice Vertex Superalgebras I: Presentation of the Principal Subspace, {\em Communications in Algebra.}{Volume 42, Issue 3} (2014), 933-961.

\bibitem[PS]{PS} M. Penn and C. Sadowski, Vertex-algebraic structure 	of principal subspaces of basic $D_4^{(3)}$-modules, to appear.

\bibitem[S1]{S1} C. Sadowski, Presentations of the principal subspaces
of the higher-level standard $\widehat{\mathfrak{sl}(3)}$-modules, 
{\em J. Pure Appl. Algebra} {\bf 219} (2015) 2300-2345.

\bibitem[S2]{S2} C. Sadowski, Principal subspaces of higher-level standard
$\widehat{\mathfrak{sl}(n)}$-modules, {\em International Journal of Mathematics},
 Vol. 26, No. 08, 1550063 (2015).

\bibitem[Za]{Za} D. Zagier, {\em The dilogarithm function}, in Frontiers in Number Theory, Physics, and Geometry II, Springer (2007), 3-65.

\bibitem[T1]{T1} 
G. Trup\v{c}evi\'{c}, 
{\em Combinatorial bases of Feigin-Stoyanovsky's type subspaces of level 1 standard
 modules for $\widetilde{\mathfrak{sl}}(l+1,\mathbb{C})$}, 
Comm. Algebra \textbf{38} (2010), 3913--3940.

\bibitem[T2]{T2} 
G. Trup\v{c}evi\'{c}, 
{\em Combinatorial bases of Feigin-Stoyanovsky's type subspaces of higher-level
 standard $\widetilde{\mathfrak{sl}}(l+1,\mathbb{C})$-modules}, 
J. Algebra \textbf{322} (2009), 3744--3774.

\bibitem[T3]{T3}
G. Trup\v{c}evi\'{c}, 
{\em Characters of Feigin-Stoyanovsky's type subspaces of level one modules
 for affine Lie algebras of types $A_{l}^{(1)}$ and $D_4^{(1)}$}, 
Glas. Mat. Ser. III 46, \textbf{66} (2011), 49--70.

\end{thebibliography}
\end{document}